\documentclass[12 pt,a4paper]{article}
\usepackage{times}
\usepackage{amssymb}
\usepackage{mathrsfs}
\usepackage{amsthm}
\usepackage{enumerate} 
\usepackage{tensor} 
\usepackage[margin=1. in]{geometry}
\usepackage{mathtools}  
\usepackage{multirow}
\usepackage{multicol}
\usepackage{setspace}
\usepackage{algorithm} 
\usepackage{algpseudocode}
\usepackage{footmisc}
\usepackage{authblk}

\newcommand{\e}{\mathbf{e}}
\newcommand{\x}{\mathbf{x}}
\newcommand{\y}{\mathbf{y}}
\newcommand{\z}{\mathbf{z}}

\newcommand{\ie}{{i.e., }}

\newtheorem{Definition}{Definition}
\newtheorem{Theorem}{Theorem}

\newtheorem{Lemma}{Lemma}

\newtheorem{Remark}{Remark}

\newtheorem{Corollary}{Corollary}

\begin{document}          
\title{Relationship between Granger non-causality and graph structure of state-space representations}

\author[*]{M. J\'{o}zsa}
\author[**]{M. Petreczky}
\author[*]{M. K. Camlibel}
\affil[*]{Johann Bernoulli Institute for Mathematics and Computer Science, University of Groningen}
\affil[**]{Centre de Recherche en Informatique, Signal et Automatique de Lille (CRIStAL)}

	\maketitle
\section*{Introduction}
A simple example for interconnections of subsystems is a cascade interconnection between two subsystems,when the information goes only from one to the other. A more general but still simple example is a coordinated system structure. Suppose that there are $n$ subsystems, of which $n-1$ are the so called \emph{agents} and one subsystem is the so called \emph{coordinator}. Then in a coordinated system only the coordinator sends information to the agents who do not share information with each other. The cascade interconnection of two subsystems is just the case when $n=2$. In this paper we consider linear stochastic systems with coordinated interconnection. The interconnection is described by the structure of the state-space representation. Our inquiry is whether these systems can be characterized by non-causality conditions on the components of the output process.

In the deterministic case it is well studied that there is always a transformation to obtain coordinated state-space representation of a system \cite{KeRaSch14,Kempker12} which in its system matrices shows the relation between the subsystems. However, it is not straightforward how to extend it to the stochastic case. In this paper we define linear stochastic state-space representation in coordinated form by keeping the structure of the system matrices like in the deterministic case.
We present results on realization theory of linear stochastic systems in coordinated form. More precisely, we present necessary and sufficient conditions for an output process $\y=[\y^T_1,\ldots,\y^T_{n-1},\y^T_n]^T$ to be the output of a minimal linear stochastic system in coordinated form with a forward innovation noise process, such that $\y_1,\ldots,\y_{n-1}$ are the outputs of the first, second, etc. $(n-1)$th agent, and $\y_n$ is the output of the coordinator. We rely on classical stochastic realization theory and on the notion of Granger non-causality. Our results for $n=2$ yield a characterization of Granger non-causality between two processes in terms of the properties of their joint linear state-space representation. 

The results of the paper could be of interest for reverse engineering the network structure of state-space representations 
which arise in system biology and neuroscience, 
\cite{GongalvesRecent2015,Yuan20111230,nordling,julius-cp,sontag_network3,roebroeck2011,sosa2011}. The results could also be useful 
for structure preserving model reduction and possibly for control design of coordinated systems. The first step towards reverse engineering the network structure is to
understand when the observed behaviour could, in principle, be realized by a state-space representation with a specific network structure. 
The same holds for structure preserving model reduction: it is useful to know whether a minimal (in terms of the dimension of states) realization admitting a 
specific network structure is for a certain behaviour. For the design of interconnected systems, understanding the relationship between
the observed behaviour and the network structure could open up the possibility of choosing alternative network structures realizing the same
functionality. The motivation for studying coordinated linear systems is that their network structure
is a simple but natural one, occurring in many applications \cite{JHvanSchuppen2014}. 

The need to understand the relationship between the observed behaviour and the network structure of linear systems is an active research area, see for example  \cite{GongalvesRecent2015,Yuan20111230,nordling}.
However, none of the
cited work addressed linear stochastic systems.  Causality relationship between time series is an  established research topic in econometrics, neuroscience and
control theory. This relationship can be characterized
in terms of the 
network structure of input-output representations
of these processes,  
see \cite{Eichler12,DufRen98,Hsiao82, Caines76,GevAnd82} and the references therein. If there is one agent, then our results can be viewed as a counterparts of the cited
papers for state-space representations. In fact, for $n=2$ Granger causality for state-space representation was studied by using transfer function approach \cite{BarSet15}. In contrast with the results in \cite{BarSet15} we give a state-space characterization for Granger non-causality by constructing it, by choosing a state process for which the system matrices are in specific form. 
The papers \cite{CaDeWy03,CaiWyn07,Caines09} are the closest ones to this paper, they provide necessary and
sufficient conditions for the existence of a linear state-space realization in the so called
conditional orthogonal form. Conditionally orthogonal state-space realizations represent a 
specific subclass of coordinated linear stochastic systems, and the conditions for the existence of such a system
are much stronger than the conditions proposed in this paper. Note that \cite{CaDeWy03,CaiWyn07,Caines09} presented conditions for
existence of conditionally orthogonal state-space representations, but in contrast to this paper, \cite{CaDeWy03,CaiWyn07,Caines09} did not address their minimality. 

Coordinated linear systems for deterministic case were studied in \cite{KeRaSch14,Kempker12,JHvanSchuppen2014}. 
In \cite{KeRaSch14,Kempker12} a general method was presented to transform a system into coordinated form. In \cite{Pambakian11, Kempker12} also Gaussian coordinated systems were studied and their LQG control.
In this paper we deal with linear stochastic systems (not necessarily Gaussian) and we observe the existence of a linear state-space representation in coordinated form in terms of the causal properties of the output processes.

The structure of the paper is the following: in section \ref{GnC} we introduce the results for $n=2$, when besides the coordinator there is one agent. We state that Granger non-causality is equivalent with a Kalman representation in block triangular form with an additional minimum phase property. For the calculation of this representation we formulate two algorithms. Besides, another state-space equivalence is provided, which is in use in the sequel, where the conditions are restricted for the coordinator subsystem.
In section \ref{Coo} we define the coordinated stochastic linear state-space representation of a process. Furthermore, we present our result for the general case when besides the coordinator there are $n-1$ agents. The proof is based on the results for the $n=2$ case and, therefore, it gives a construction and an algorithm for the calculation of a state-space representation in coordinated form. In section \ref{Examples} we provide examples for the algorithms calculating the representations in block triangular and coordinated form. Finally, the proofs of the theorems and lemmas can be found in section \ref{Proofs}.

\subsection*{Preliminaries}
We use the standard notation and terminology of probability theory. All the processes considered in this paper are discrete-time processes, whose values are vectors with real entries. 
The discrete-time axis is the set of integers
$\mathbb{Z}$.

The space of square-integrable random variables forms a Hilbert space ($\mathcal{H}$) with the covariance function as the inner product. The Hilbert space generated by the one-dimensional components of a process $\z$ at each time is denoted by $\mathcal{H}^{\z}$ and, similarly, for a joint process $[\y^T,\z^T]^T$ it is denoted by $\mathcal{H}^{\y,\z}$. In addition, we write  $\mathcal{H}^{\z}_{t}$, $\mathcal{H}^{\z}_{t-}$ and  $\mathcal{H}^{\z}_{t+}$ for the Hilbert spaces generated by the one-dimensional components of $\z(t)$, $\lbrace \z(s)\rbrace_{s=-\infty}^{t-1}$ and $\lbrace \z(s)\rbrace_{s=t}^{\infty}$, respectively. 
The orthogonal projection of $A\subset \mathcal{H}$ onto $B \subset \mathcal{H}$ is written as $E_l[A|B]:=\lbrace E_l[a|B], a \in A \rbrace$. The orthogonality of a process $\y$ to $A\subset \mathcal{H}$ is meant by element-wise and is written as $\y \perp A$. Note that for Gaussian processes the best prediction is the linear prediction and thus the orthogonal projection is equivalent with the conditional expectation. We write the sum of two subspaces $A,B\subseteq \mathcal{H}$ as $A+B$, the orthogonal direct sum of them as $A \oplus B$ and the orthogonal complement of $B$ in $A$ as $A \ominus B$.	
For closed subspaces $A,B,C \subseteq \mathcal{H}$, we say that $A,B$ have a \emph{conditionally trivial intersection}   with respect to  $C$,
denoted by $A\cap B| C=\lbrace 0 \rbrace $,
if 
$$\{ a-E_l[a |C] \mid a \in A\} \cap \{ b - E_l[b | C] \mid b \in B\}=\{0\},$$
i.e. the intersection of the projections of $A$ and $B$  
onto the orthogonal complement of $C$ in $\mathcal{H}$ is the zero subspace. 

\section{Characterization of Granger non-causality}\label{GnC}
The concept of Granger causality between two discrete random processes \cite{Granger63} turned out to be an important and useful tool for time series analysis in neuroscience and economics. In this section we observe Granger non-causality in terms of state-space representation of a process. This notion can be explained the following way: taking a joint process $\y=[\y_1^T, \y_2^T]^T$ we say that $\y_1$ is Granger noncausal for $\y_2$ if for all $k \geq 0$ the best $k$-step linear prediction of $\y_2$ based on the past values of $\y_2$ is the same than based on the past of $\y$. We show that Granger non-causality between $\y_1$ and $\y_2$ is equivalent with a (forward) innovation representation\footnote{In this paper innovation is always meant as forward innovation.} for $\y$ having block triangular system matrices.

From now on we assume that the processes are weakly-stationary, full-rank and purely non-deterministic. The state-space representations of the processes are meant to be stable. Furthermore, $I$ is always meant to be the identity matrix with the appropriate dimension. To begin with, we introduce an abbreviation (ZMSIR) for the processes which have finite-dimensional linear state-space representation and then we define Granger non-causality for these processes. After introducing the Kalman representation of a process we present equivalences of Granger non-causality in linear state-space representation. Based on the proofs (Section \ref{Proofs}) we provide algorithms which calculate a state-space representation with block triangular matrices characterizing Granger non-causality in the output process.

\begin{Definition}[ZMSIR]
	A stochastic process $\z(t) \in \mathbb{R}^k, t \in \mathbb{Z}$ is called zero-mean square-integrable with rational spectrum (abbreviated by ZMSIR), if it is square-integrable, the expectations $E[\z(t)]$ exist and equals zero and the spectral density $f(\omega)$ is a rational function of $e^{-i\omega}$.
\end{Definition}

\begin{Definition}[Granger non-causality]
	Consider a $\y=[\y_1^T, \y_2^T]^T$ ZMSIR process. We say that $\y_1$ is Granger noncausal for $\y_2$, if for all $t,k \in \mathbb{Z}$,
	\[ E_l[\y_2(t+k) \mid \mathcal{H}_{t-}^{\y_2}] = E_l[\y_2(t+k) \mid \mathcal{H}_{t-}^{\y}]. \]
\end{Definition}\vspace{.2cm}
In this work we study linear time-invariant state-space representations of ZMSIR processes which is defined for a $\y$ process as
\begin{align*}
\x(t+1) & =A\x(t)+B\e(t)\\
\y(t) & =C\x(t)+D\e(t)
\end{align*}
where for any $t,k \in \mathbb{Z}$, $k \geq 0$, $E[\e(t)\e^T(t-k-1)]=0$, $E[\e(t)\x^T(t-k)]=0$, \ie $\e(t)$ is white noise and uncorrelated with $\x(t-k)$. We say that this representation is in \textit{innovation form} if $\e(t)=\y(t)-E_l[\y(t)|\mathcal{H}^{\y}_{t-}]$ $\forall t \in \mathbb{Z}$. There is a specific state-space representation in innovation form called the Kalman representation. We say that a linear state-space representation
	\begin{align*}
	\x(t+1) &= A\x(t) +K\e(t)\\
	\y(t) &= B\x(t)+ \e(t)
	\end{align*}
is a \textit{Kalman representation} if $K$ is the Kalman gain, $\e(t)$ is the innovation process of $\y$ and $\x(t)$ is in the space spanned by the past of $\y$, $\mathcal{H}^{\y}_{t-}$. In this section the observed systems occur to have matrices in block triangular form. For this purpose, we introduce a notion for a representations being in block triangular form.
\begin{Definition}
	Consider a $\y=[\y_1^T,\y_2^T]^T$ ZMSIR process with a state-space representation $(A,B,C,D)$. We say that the system is in block triangular form if it can be written as
	\begin{align}\label{Triang}
	\begin{split}
	\begin{bmatrix} \x_1(t+1) \\ \x_2(t+1) \end{bmatrix} &=
	\begin{bmatrix}A_{1,1} & A_{1,2} \\ 0 & A_{2,2} \end{bmatrix}
	\begin{bmatrix} \x_1(t) \\ \x_2(t) \end{bmatrix} + 
	\begin{bmatrix}B_{1,1} & B_{1,2} \\ 0 & B_{2,2} \end{bmatrix} \begin{bmatrix} \e_1(t) \\ \e_2(t) \end{bmatrix} \\ 
	\begin{bmatrix} \y_1(t) \\ \y_2(t) \end{bmatrix} &=
	\begin{bmatrix}C_{1,1} & C_{1,2} \\ 0 & C_{2,2} \end{bmatrix}
	\begin{bmatrix} \x_1(t) \\ \x_2(t) \end{bmatrix} + \begin{bmatrix}D_{1,1} & D_{1,2} \\ 0 & D_{2,2} \end{bmatrix}\begin{bmatrix} \e_1(t) \\ \e_2(t) \end{bmatrix},
	\end{split}
	\end{align}
 where for $i,j=1,2$, $i\leq j$ the submatrices $A_{i,j} \in \mathbb{R}^{p_i \times p_j}$, $B_{i,j} \in \mathbb{R}^{p_i \times q_j}$, $C_{i,j} \in \mathbb{R}^{r_i \times p_j}$, $D_{i,j} \in \mathbb{R}^{r_i \times q_j}$ for some $p_1,q_i,r_i >0 $ and $p_2\geq 0$.
\end{Definition}

The next theorem is about a characterization of Granger non-causality by specifying a Kalman representation to be in block triangular form.

\begin{Theorem} \label{Kalman22}
	Consider a $\y=[\y_1^T,\y_2^T]^T$ ZMSIR process. Then $\y_1$ does not Granger cause $\y_2$ if and only if there exists a minimal Kalman representation in block triangular form
		\begin{align}\label{Km22ab}
		\begin{split}
		\begin{bmatrix} \x_1(t+1) \\ \x_2(t+1) \end{bmatrix} &=
		\underbrace{\begin{bmatrix}A_{1,1} & A_{1,2} \\ 0 & A_{2,2} \end{bmatrix}}_{A}
		\begin{bmatrix} \x_1(t) \\ \x_2(t) \end{bmatrix} + 
		\underbrace{\begin{bmatrix}K_{1,1} & K_{1,2} \\ 0 & K_{2,2} \end{bmatrix}}_{K} \underbrace{\begin{bmatrix} \e_1(t) \\ \e_2(t) \end{bmatrix}}_{\e}  \\ 
		\begin{bmatrix} \y_1(t) \\ \y_2(t) \end{bmatrix} &=
		\underbrace{\begin{bmatrix}C_{1,1} & C_{1,2} \\ 0 & C_{2,2} \end{bmatrix}}_{C}
		\begin{bmatrix} \x_1(t) \\ \x_2(t) \end{bmatrix} + \begin{bmatrix} \e_1(t) \\ \e_2(t) \end{bmatrix},
		\end{split}
		\end{align}
	where the subsystem $(A_{2,2},K_{2,2},C_{2,2},I)$ representing $\y_2$ is minimum phase.
\end{Theorem}

Recall that a MIMO system is \textit{minimum phase} if it has no zeros outside the unit circle which is equivalent with the existence of a stable causal inverse system. The representation in Theorem \ref{Kalman22} is a special innovation representation showing the causal relation between $\y_1$ and $\y_2$. From system theory we know that all minimal linear systems realizing $\y$ in innovation form are isomorphic (\cite[Theorem 6.6.1]{LinPic15}). Therefore, given a minimal state-space representation in innovation form we can transform it into block triangular form if and only if the Granger non-causality condition holds. Correspondingly, given any representation of a process we can calculate a system in innovation form. 

\begin{Corollary}
	Consider a $\y=[\y_1^T,\y_2^T]^T$ ZMSIR process. If $\y_1$ does not Granger cause $\y_2$ then from any $(A,B,C,D)$ state-space representation of $\y$ a Kalman representation in block triangular form can be calculated.
\end{Corollary}

The necessary part of Theorem \ref{Kalman22} can be easily seen if we consider that Granger non-causality is equivalent with the block triangular form of the Wold decomposition (\cite{CaiCha75}, \cite{Caines76}). Kalman representation of a $\y$ process in terms of non-causality was observed by Barnett and Seth (\cite{BarSet15}). They pointed out that an equivalent condition for the Wold decomposition being in block triangular form, i.e. condition for non-causality, simplifies to the following:
\begin{equation} \label{BarSet}
\left(C(A-KC)^kK\right)_{21}=0, \hspace{.2cm}\forall k\geq 0,
\end{equation}
where $(.)_{21}$ is the left lower block of the matrix and $(A,K,C)$ are the system matrices of a Kalman representation. The main difference between their result and Theorem \ref{Kalman22} is that our statement is about the block triangular structure of the system matrices of a Kalman representation and in the proof we construct the state-space model \eqref{Km22ab} which naturally satisfies \eqref{BarSet}. Next, we present two algorithms to calculate the system matrices $(A,K,C,I)$ in \eqref{Km22ab}. The first takes a minimal state-space representation as its input and transforms it into a block triangular Kalman representation while the second calculates the same system from covariances of the output. Examples are provided thereafter in section \ref{Examples} where we use empirical covariances as the input of the second algorithm.
\begin{algorithm}[H]\vspace{.2cm}
	\textbf{Input} ($A,B,C,D,Q,ind$): minimal state-space representation for $\y$ with noise covariance and an index set for $\y_2$\\
	\textbf{Output} ($A_k,K,C_k$): system matrices of \eqref{Km22ab} in Theorem \ref{Kalman22}\vspace{.1cm}
	\caption{Kalman representation in block triangular form from system matrices}
	\begin{algorithmic}[0] 
		\State \textbf{Step 1} Solve $P=APA^T+BQ_{\e}B^T$ with $Q_{\e}=E[\e(t)\e^T(t)]$ and define $\bar{C}:=CPA^T+DB^T$.
		
		\State \textbf{Step 2} If $(C_{ind,.},A)$ is observable then define the matrices $(A_k,C_k,\bar{C}_k):=(A,C_{ind,.},\bar{C}_{ind,.})$. Otherwise calculate a transformation matrix $T$ such that $(C_{ind,.},A)$ is in Kalman observable form and define the matrices $(A_k,C_k,\bar{C}_k):=(T A T^{-1},C T^{-1},T \bar{C})$. 
		
		\State \textbf{Step 3} Calculate the covariance of the Kalman state process by solving the DARE		
		$$ X = A_k X A_k^T + (\bar{C}_k-A_k X C_k^T)(\Lambda_0-C_k X C_k^{T})^{-1}(\bar{C}_k-A_k X C_k^T)^T.$$		
		
		\State \textbf{Step 4} Define the Kalman gain as $K:=(\bar{C}_k^{T}-A_k X C_k^{T})(\Lambda_0-C_k X C_k^{T})^{-1}$.
		
		\State \textbf{Step 5} If $(C_{ind,.},A)$ was observable then Granger non-causality holds. If not, then Granger non-causality holds if and only if every matrix in $(A_k, K, C_k)$ has block triangular form with appropriate dimensions.
	\end{algorithmic}
	\label{algoGrangerABCD}
\end{algorithm}
Algorithm 2 is based on the covariances of the output process which opens up the possibility to calculate system \eqref{Km22ab} directly from data by using empirical covariances. In this case there is freedom of accuracy in the choice of $M$ and the tolerance for numerical nonzero numbers. Note that $M$ should be larger than equal than the McMillan degree of an innovation representation of $\y$. It is worth to mention that in the proof of Theorem 1 the case when $(C_{ind},A)$ is observable in step 2 of Algorithm 1 turns out to happen exactly when the dimension of $\x_1$ is zero in system \eqref{Km22ab}.

\begin{algorithm}[H]\vspace{.2cm}
	\textbf{Input} ($\Lambda_k,ind$): Markov parameters of $\y$ and an index set for $\y_2$\\
	\textbf{Output} ($A_k,K,C_k$): system matrices of \eqref{Km22ab} in Theorem \ref{Kalman22}\vspace{.1cm}
	\caption{Kalman representation in block triangular form from output covariances}
	\begin{algorithmic}[0] 
		\State \textbf{Step 1} Define the Hankel matrix and the shifted Hankel matrix as
		$$H_0=\begin{bmatrix}
		\Lambda_1 & \Lambda_2 & \ldots & \Lambda_M \\
		\Lambda_2 &  \Lambda_3 &\ldots & \Lambda_{M+1} \\
		\vdots & \vdots & & \vdots \\
		\Lambda_M & \Lambda_{M+1} & \ldots & \Lambda_{2M-1}
		\end{bmatrix}, \hspace{.4cm} H_1=\begin{bmatrix}
		\Lambda_2 & \Lambda_3 & \ldots & \Lambda_{M+1} \\
		\Lambda_3 &  \Lambda_4 &\ldots & \Lambda_{M+2} \\
		\vdots & \vdots & & \vdots \\
		\Lambda_{M+1} & \Lambda_{M+2} & \ldots & \Lambda_{2M}
		\end{bmatrix},$$
		where $M\geq n/m$ for $n$ being the (assumed) McMillan degree and $m=\dim(\y)$.
		\State \textbf{Step 2} Take the SVD of the Hankel matrix $H_0=USV^T $ and estimate the state-space dimension $n$ by looking at the number of (numerical) nonzero singular values. Let $(U_n,S_n,V_n)$ be the SVD of the $n$-rank approximation of $H_0$ and define the matrices
		$$A:=S_n^{-1/2}U_n^TH_1V_nS_n^{-1/2}, \hspace{.25cm} C:=(U_nS_n^{1/2})_{1:m,.}, \hspace{.25cm} \bar{C}:=(VS^{1/2})_{1:m,.}.$$
		\State \textbf{Step 3} Step 2-3-4-5 of Algorithm 1.
	\end{algorithmic}
	\label{algoGrangerY}
\end{algorithm}

\begin{Lemma}[Correctness of Algorithm 1 and 2]\label{Algo12}
	Consider a $\y=[\y_1^T,\y_2^T]^T$ ZMSIR process with Markov parameters $\Lambda_k^{\y} $ and a minimal state-space representation $(A,B,C,D,Q)$ where $Q$ is the input covariance. The system $(A_k,K,C_k,I)$ where $(A_k,K,C_k)$ is either the output of Algorithm 1 with input $(A,B,C,D,Q)$ or the output of Algorithm 2 with input $\lbrace \Lambda_k^{\y} \rbrace_{k=1}^{2M}$ ($M \geq n/m$ where $A\in \mathbb{R}^{n \times n}$ and $m=\dim(\y)$) is in block triangular form if and only if $\y_1$ does not Granger cause $\y_2$.
\end{Lemma}
In Theorem \ref{Kalman22} we saw that a Kalman representation in block triangular form with a minimum phase condition characterizes Granger non-causality. However, in the next section the matrix $K$ is not going to be the Kalman gain. Therefore we need a generalized form of Theorem \ref{Kalman22} in the sense that instead of assuming system \eqref{Km22ab} to be a Kalman representation we dispense conditions only for the subsystem representing $\y_2$.

\begin{Theorem} \label{GC22}
	Consider a $\y=[\y_1^T , \y_2^T]^T$ ZMSIR process. Then $\y_1$ does not Granger cause $\y_2$ if and only if there exists a state-space representation of $\y$ in the form of
	\begin{align} \label{GC22ab}
	\begin{split}
	\begin{bmatrix} \x_1(t+1) \\ \x_2(t+1) \end{bmatrix} &=
	\underbrace{\begin{bmatrix}A_{1,1} & A_{1,2} \\ 0 & A_{2,2} \end{bmatrix}}_{A}
	\begin{bmatrix} \x_1(t) \\ \x_2(t) \end{bmatrix} + 
	\underbrace{\begin{bmatrix}K_{1,1} & K_{1,2} \\ 0 & K_{2,2} \end{bmatrix}}_{K} \begin{bmatrix} \e_1(t) \\ \e_2(t) \end{bmatrix}  \\
	\begin{bmatrix} \y_1(t) \\ \y_2(t) \end{bmatrix} &=
	\underbrace{\begin{bmatrix}C_{1,1} & C_{1,2} \\ 0 & C_{2,2} \end{bmatrix}}_{C}
	\begin{bmatrix} \x_1(t) \\ \x_2(t) \end{bmatrix} + \begin{bmatrix} \e_1(t) \\ \e_2(t) \end{bmatrix}, 
	\end{split}
	\end{align}
	such that the following holds:
	\begin{enumerate}[(i)]
		\item \label{GC22i} $\e_2(t)=\y_2(t)-E_l[\y_2(t)|\mathcal{H}^{\y}_{t-}]$;	
		\item \label{GC22ii} the matrix $Q_{\x_2}:=E[\x_2(t)\x_2^T(t)]$ is the minimal positive definite solution of
		\begin{equation} \label{Km22gain}
			\Sigma=A_{2,2}\Sigma A_{2,2}^T+(\bar{C}_{2,2}-A_{2,2}\Sigma C_{2,2}^T)( \Lambda_0^{\y_2}- C_{2,2}\Sigma C_{2,2}^T)^{-1}(\bar{C}_{2,2}-A_{2,2}\Sigma C_{2,2}^T)^T.\footnote{The existence of the inverse comes from the full-rank property of $\y$.}
		\end{equation}
		\item \label{GC22iii} the subsystem representing $\y_2$ is minimal.
	\end{enumerate}
\end{Theorem}
As we mentioned before, this characterization is advantageous in section 3 where we combine different representations into a so-called coordinated form. In fact, we specify only $K_{2,2}$ (there $K_{n,n}$) to be a Kalman gain and $K_{i,j}$ in general comes from the Kalman gain for appropriate subsystem.

\section{Coordinated systems}\label{Coo}
A coordinated system determines the directions of the communication between subsystems. Suppose that there are $n$ subsystems, where one is called the coordinator and every other is called agent. Then the coordinated structure restricts the communication flow in such a way that the coordinator can send information to the other subsystems and no other direction of communication is allowed. 

In this section we introduce stochastic linear state-space representations of a process being in coordinated form. In the light of the previous section we construct a state-space representation for an output process $\y=[\y_1^T,\ldots,\y_{n-1}^T, \y_n^T]^T$ with coordinated system structure, supposing that certain non-causal relations are present in $\y$. By this, necessary and sufficient conditions are provided for the existence of an innovation representation in coordinated form. The conditions are based on the terms of Granger non-causality and conditional Granger non-causality. The proof is constructive, thus, if a $\y$ process satisfies the required non-causality conditions then a stochastic linear state-space representation in coordinated form can be calculated algorithmically. In addition, a condition for the minimality of the constructed system is provided. \\
\begin{Definition}[conditional Granger non-causality]
	Consider a ZMSIR process $\y=[\y_1^T, \y_2^T, \y_3]^T$. We say that $\y_1$ is conditionally Granger noncausal for $\y_2$ with respect to $\y_3$, if for all $t,k \in \mathbb{Z}$,
	\[ E_l[\y_2(t+k) \mid \mathcal{H}_{t-}^{\y_2,\y_3}] = E_l[\y_2(t+k) \mid \mathcal{H}_{t-}^{\y}]. \]
\end{Definition}

\begin{Definition}[Stochastic linear state-space representation in coordinated form]
	A stochastic linear state-space representation of a $\y=[\y_1^T,\dots,\y_{n-1}^T,\y_n^T]^T$ process is in coordinated form if it is written as
	\begin{multicols}{2}
		\resizebox{1.88 \linewidth}{!}
		{\hskip -5 mm
		\begin{minipage}{2\linewidth}
			\begin{align*} 
			\begin{split} 
			\left[ \! \begin{array}{c} \x_1(t+1) \\ \x_2(t+1) \\ \vdots \\ \x_{n-1}(t+1) \\ \hline \x_n(t+1) \end{array}\!\right] \! & \mathord{=} \!
			\underbrace{\left[ \! \begin{array}{cccc|c} A_{1,1} & 0 & \dots & 0 & A_{1,n} \\ 0 & A_{2,2} & \dots & 0 & A_{2,n} \\ \vdots & \vdots & \ddots & \vdots & \vdots\\ 0 & 0 & \dots & A_{n-1,n-1} & A_{n-1,n} \\\hline 0 & 0 & \dots & 0 & A_{n,n}\end{array}\!\right]}_A\!
			\underbrace{\left[ \! \begin{array}{c}  \x_1(t) \\ \x_2(t)  \\ \vdots \\ \x_{n-1}(t)  \\ \hline \x_n(t) \end{array}\!\right]}_{\x(t)}
			\mathord{+} \!
			\underbrace{\left[\! \begin{array}{cccc|c} B_{1,1} & 0 & \dots & 0 & B_{1,n} \\ 0 & B_{2,2} & \dots & 0 & B_{2,n} \\ \vdots & \vdots & \ddots & \vdots & \vdots\\ 0 & 0 & \dots & B_{n-1,n-1} & B_{n-1,n} \\	\hline 0 & 0 & \dots & 0 & B_{n,n}\end{array} \! \right]}_B \!
			\underbrace{\left[ \! \begin{array}{c}  \e_1(t)  \\ \e_2(t)  \\ \vdots \\ \e_{n-1}(t)  \\ \hline \e_n(t)  \end{array}\!\right]}_{\e(t)} \\ 
			\underbrace{\left[ \! \begin{array}{c}  \y_1(t)  \\ \y_2(t)  \\ \vdots \\ \y_{n-1}(t)  \\ \hline \y_n(t)  \end{array}\!\right]}_{\y(t)} & \mathord{=}
			\underbrace{\left[\!\begin{array}{cccc|c} C_{1,1} & 0 & \dots & 0 & C_{1,n} \\ 0 & C_{2,2} & \dots & 0 & C_{2,n} \\ \vdots & \vdots & \ddots & \vdots & \vdots\\ 0 & 0 & \dots & C_{n-1,n-1} & C_{n-1,n} \\	\hline 0 & 0 & \dots & 0 & C_{n,n}\end{array}\! \right]}_C \!
			\left[ \! \begin{array}{c}  \x_1(t)  \\ \x_2(t)  \\ \vdots \\ \x_{n-1}(t)  \\ \hline \x_n(t) \end{array}\!\right] \!  +  \!
			\underbrace{\left[\! \begin{array}{cccc|c} D_{1,1} & 0 & \dots & 0 & D_{1,n} \\ 0 & D_{2,2} & \dots & 0 & D_{2,n} \\ \vdots & \vdots & \ddots & \vdots & \vdots\\ 0 & 0 & \dots & D_{n-1,n-1} & D_{n-1,n} \\	\hline 0 & 0 & \dots & 0 & D_{n,n}\end{array} \! \right]}_D\! \left[ \! \begin{array}{c}  \e_1(t)  \\ \e_2(t)  \\ \vdots \\ \e_{n-1}(t)  \\ \hline \e_n(t)  \end{array}\!\right]\!
			\end{split}	
			\end{align*}
		\end{minipage}}\break 
		\vskip 12.5mm
			\begin{equation}
			\label{coordrep}
			\end{equation}
	\end{multicols}
	where the processes $\x_i,\e_i,\y_i$ take values in the spaces $\mathbb{R}^{p_i}$, $\mathbb{R}^{q_i}$, $\mathbb{R}^{r_i}$, respectively, for some positive integers $p_i,q_i,r_i$ and $A_{i,j} \in \mathbb{R}^{p_i \times p_j}$, $B_{i,j} \in \mathbb{R}^{p_i \times q_j}$, $C_{i,j} \in \mathbb{R}^{r_i \times p_j}$, $D_{i,j} \in \mathbb{R}^{r_i \times q_j}$, $i,j\in \lbrace 1,\ldots,n-1,n\rbrace$. In addition, we say that a system matrix is in coordinated form if it has the same zero structure as in the system matrices above.
\end{Definition}
	
The latter definition is based on the deterministic terminology (\cite{Kempker12},\cite{JHvanSchuppen2014}) and on the definition of Gaussian coordinated systems \cite{Pambakian11, Kempker12}. The term \textit{coordinated} is used because the information flow is restricted in such a way that $\y_i \in \mathcal{H}^{\e_i,\e_n}_{(t+1)-}$ for $i \in \lbrace 1,2,\dots, n-1 \rbrace$ and $\y_n \in \mathcal{H}^{\e_n}_{(t+1)-}$.

In our main result we construct an innovation representation in coordinated form using Granger non-causality conditions. Applying Theorem \ref{Kalman22} for some partitions of the processes which satisfy a Granger non-causality condition we can combine the resulting subsystems into coordinated form. Moreover, a condition for minimality arises from the construction.

\begin{Theorem} \label{GCnn}
	Consider a $\y=[\y_1^T,\dots,\y_{n-1}^T, \y_n^T]^T$ ZMSIR process. Then both
	\begin{enumerate}[(i*)]
		\item \label{GC1ni} $\y_i$ is Granger noncausal for $\y_n$, $i \in \lbrace 1, 2, \ldots, n-1 \rbrace$
		\item \label{GC1nii} $\y_i$ is conditionally Granger noncausal for $\y_j$ with respect to $\y_n$ $i,j \in \lbrace 1, 2, \ldots, n-1 \rbrace$, $i \neq j$
	\end{enumerate}
	if and only if there exists an innovation representation in the form of
	\begin{multicols}{2}
		\resizebox{1.88 \linewidth}{!}
		{\hskip -9 mm
			\begin{minipage}{2\linewidth}
				\begin{align*} 
				\begin{split} 
				\left[ \! \begin{array}{c} \x_1(t+1) \\ \x_2(t+1) \\ \vdots \\ \x_{n-1}(t+1) \\ \hline \x_n(t+1) \end{array}\!\right] \! & \mathord{=} \! \underbrace{\left[ \! \begin{array}{cccc|c} A_{1,1} & 0 & \dots & 0 & A_{1,n} \\ 0 & A_{2,2} & \dots & 0 & A_{2,n} \\ \vdots & \vdots & \ddots & \vdots & \vdots\\ 0 & 0 & \dots & A_{n-1,n-1} & A_{n-1,n} \\\hline 0 & 0 & \dots & 0 & A_{n,n}\end{array}\!\right]}_A\! \underbrace{\left[ \! \begin{array}{c}  \x_1(t) \\ \x_2(t)  \\ \vdots \\ \x_{n-1}(t)  \\ \hline \x_n(t) \end{array}\!\right]}_{\x(t)} \mathord{+} \! \underbrace{\left[\! \begin{array}{cccc|c} B_{1,1} & 0 & \dots & 0 & B_{1,n} \\ 0 & B_{2,2} & \dots & 0 & B_{2,n} \\ \vdots & \vdots & \ddots & \vdots & \vdots\\ 0 & 0 & \dots & B_{n-1,n-1} & B_{n-1,n} \\	\hline 0 & 0 & \dots & 0 & B_{n,n}\end{array} \! \right]}_B \! \underbrace{\left[ \! \begin{array}{c}  \e_1(t)  \\ \e_2(t)  \\ \vdots \\ \e_{n-1}(t)  \\ \hline \e_n(t)  \end{array}\!\right]}_{\e(t)} \\ 
				\underbrace{\left[ \! \begin{array}{c}  \y_1(t)  \\ \y_2(t)  \\ \vdots \\ \y_{n-1}(t)  \\ \hline \y_n(t)  \end{array}\!\right]}_{\y(t)} & \mathord{=}
				\underbrace{\left[\!\begin{array}{cccc|c} C_{1,1} & 0 & \dots & 0 & C_{1,n} \\ 0 & C_{2,2} & \dots & 0 & C_{2,n} \\ \vdots & \vdots & \ddots & \vdots & \vdots\\ 0 & 0 & \dots & C_{n-1,n-1} & C_{n-1,n} \\	\hline 0 & 0 & \dots & 0 & C_{n,n}\end{array}\! \right]}_C \!
				\left[ \! \begin{array}{c}  \x_1(t)  \\ \x_2(t)  \\ \vdots \\ \x_{n-1}(t)  \\ \hline \x_n(t) \end{array}\!\right] \!  +  \! \left[ \! \begin{array}{c}  \e_1(t)  \\ \e_2(t)  \\ \vdots \\ \e_{n-1}(t)  \\ \hline \e_n(t)  \end{array}\!\right]\!, \hspace{0.2cm}
				\end{split}	
				\end{align*}
			\end{minipage}}\break 
			\vskip 12.5mm
			\begin{equation}
			\label{Cor_nn}
			\end{equation}
		\end{multicols}
		where the following holds:
		\begin{enumerate}[(i)]
			\item \label{Thi} $\e_i(t)=\y_i(t)-E_l[\y_i(t)|\mathcal{H}^{\y_i,\y_n}_{t-}]$, $i \in \lbrace 1,2,\ldots,n\rbrace$
			\item \label{Thii} the matrix $Q_{\x_n}:=E[\x_n(t)\x_n^T(t)]$ is the minimal positive definite solution of
			\begin{equation}
			\Sigma=A_{n,n}\Sigma A_{n,n}^T+(\bar{C}_{n,n}-A_{n,n}\Sigma C_{n,n}^T)( \Lambda_0^{\y_n}- C_{n,n}\Sigma C_{n,n}^T)^{-1}(\bar{C}_{n,n}-A_{n,n}\Sigma C_{n,n}^T)^T
			\end{equation}
			\item \label{Thiii} the subsystem representing $\y_n$ is minimal.
		\end{enumerate}
		In addition, if the representation above exists then it is minimal if and only if
		\begin{equation} \label{Thiv}
		{E_l[\mathcal{H}^{\y_i}_{t+}|\mathcal{H}_{t-}^{\y_i,\y_n}] \cap E_l[\mathcal{H}^{\y_j}_{t+}|\mathcal{H}_{t-}^{\y_j,\y_n}] \mid E_l[\mathcal{H}^{\y_n}_{t+}|\mathcal{H}^{\y_n}_{t-}]= \lbrace 0 \rbrace} \hspace{.2cm} i,j \in \lbrace 1, 2, \ldots, n-1 \rbrace, i \neq j.
		\end{equation}
\end{Theorem}

\begin{Remark} \label{RemCGC}
	Conditions (\ref{GC1ni}*) and (\ref{GC1nii}*) imply that $\y_i$ is Granger noncausal for $[\y_j^T,\y_n^T]^T$, $i,j \in \lbrace 1, 2, \ldots, n-1 \rbrace$, $i \neq j$. Apparently, (\ref{GC1nii}*) can be changed to it hence Granger noncausality is enough to be inspected and conditional Granger non-causality can be neglected.
\end{Remark}
Similarly, then in section 2 we provide two algorithms for the calculation of \eqref{Cor_nn} in Theorem \ref{GCnn} with the assumption that (\ref{GC1ni}*) and (\ref{GC1nii}*) hold. Algorithm \ref{algoCoordABCD} takes a state-space representation as its input and transforms it into coordinated form while Algorithm \ref{algoCoordY} calculates the same system from the covariances of the output. Note that by Remark \ref{RemCGC} condition (\ref{GC1ni}*) and (\ref{GC1nii}*) can be verified using Granger non-causality tests.

\begin{algorithm}[H]\vspace{.2cm}
	\textbf{Input} $(A,B,C,D,Q,cut)$: state-space representation for $\y$ with input covariance and an index set for the partitioning $\y=[\y_1^T,\ldots,\y_n^T]^T$\\
	\textbf{Output} $(A_k,K,C_k,I)$: system matrices of the state-space representation in Theorem \ref{GCnn}\vspace{.1cm}
	\caption{Coordinated representation from system matrices}
	\begin{algorithmic}[0] 
		\State \textbf{Step 1} Minimize each subsystem representing $[\y_{cut(i)}^T, \y_{cut(n)}^T]^T$, $i \in \lbrace 1, 2, \dots , n-1 \rbrace$ and call Algorithm \ref{algoGrangerABCD}. Denote its output as $(A_{k_i}, K_i, C_{k_i}, I_i)$.
		
		\State \textbf{Step 2} For $(A_{k_i}, K_i, C_{k_i}, I_i)$ denote the index set for the coordinator part of the state-space with $\x_{c}$ and the coordinator part of the output with $\y_{c}$. Consider the subsystems for $\y_c$ with the matrices $$\left(A_{k_i}(\x_c,\overline{\x}_c), K_i(\x_c,\y_c), C_{k_i}(\y_c,\overline{\x}_c), I_i(\y_c,\y_c)\right)$$
		and define the transformations $S_i := C_{k_i}(\y_c,\overline{\x}_c)^{-1}C_{k_1}(\y_c,\overline{\x}_c)$.\footnote{ Over line denotes complement indexes.} Transform the subsystems as $(A_{tr_i}, K_{tr_i}, C_{tr_i}, I_i):=(S_iA_{k_i}S_i^{-1},K_iS_i^{-1},S_iC_{k_i},I_i)$. 
		
		\State \textbf{Step 3} Merge $(A_{tr_i}, K_{tr_i}, C_{tr_i}, I_i)$ into a system $(A_{k}, K, C_{k}, I)$ in coordinated form by allocating them into the $\left((i,i),(i,n),(n,i),(n,n)\right)$ blocks. Note that the $(n,n)$ block was transformed to be the same for all systems in Step 2 and the $(n,i)$ block is always zero.
		
		\State \textbf{Step 4} Minimality holds if $(A_k,K)$ is controllable.
	\end{algorithmic}
	\label{algoCoordABCD}
\end{algorithm}
\begin{algorithm}[H]\vspace{.2cm}
	\textbf{Input} $(\Lambda_k,cut)$: Markov parameters of the output process $\y$ and an index set for the partitioning $\y=[\y_1^T,\ldots,\y_n^T]^T$\\
	\textbf{Output} $(A_k,K,C_k,I)$: system matrices of the state-space representation in Theorem \ref{GCnn}\vspace{.1cm}
	\caption{Coordinated representation from output covariances}
	\begin{algorithmic}[0] 
		\State \textbf{Step 1}  For $\Lambda_k^{\y_i,\y_n}$, $i \in \lbrace 1, 2, \dots , n-1 \rbrace$ call Algorithm \ref{algoGrangerY} and denote its output as $(A_{k_i}, K_i, C_{k_i}, I_i)$.
		
		\State \textbf{Step 2} Step 2-3-4 of Algorithm \ref{algoCoordABCD}.
	\end{algorithmic}
	\label{algoCoordY}
\end{algorithm}	

\begin{Lemma}[Correctness of Algorithm 3 and 4]\label{Algo34}
	Consider a $\y=[\y_1^T,\dots,\y_{n-1}^T, \y_n^T]^T$ ZMSIR process with Markov parameters $\Lambda_k^{\y} $ and a state-space representation $(A,B,C,D,Q)$ where $Q$ is the input covariance. The system $(A_k,K,C_k,I)$ where $(A_k,K,C_k)$ is either the output of Algorithm 3 with input $(A,B,C,D,Q)$ or the output of Algorithm 4 with input $\lbrace \Lambda_k^{\y} \rbrace_{k=1}^{2M}$ ($M \geq n/m$ where $A\in \mathbb{R}^{n \times n}$ and $m=\dim(\y)$) is in coordinated form with the properties \eqref{Thi}-\eqref{Thii}-\eqref{Thiii} in Theorem \ref{GCnn} if and only if the conditions below hold.
	\begin{enumerate}
	\item $\y_i$ is Granger noncausal for $\y_n$, $i \in \lbrace 1, 2, \ldots, n-1 \rbrace$
	\item $\y_i$ is Granger noncausal for $[\y_j^T,\y_n^T]^T$, $i,j \in \lbrace 1, 2, \ldots, n-1 \rbrace$, $i \neq j$.
	\end{enumerate}
	In addition, $(A_k,K,C_k,I)$ is minimal if and only if \eqref{Thiv} is fulfilled.
\end{Lemma}

\section{Examples for the algorithms}\label{Examples}

\textbf{Examples for algorithm 1 and 2:} We define a state-space representation of a $\y=[\y_1^T, \y_2^T ,\y_3^T ]^T$ process $(A,B,C,I)$ with noise covariance $Q$ as given below.\\
		\resizebox{1 \linewidth}{!}
		{\begin{minipage}{1\linewidth}
			\begin{align*} 	
			&A,B,C,Q :=\\[.5em]
			& \left[\begin{array}{ccc|cc} 0.45 & 0.11 & 0.07 & 0.09 & 0.43\\ 0.6 & 0.02 & 0.64 & 0.08 & 0.27\\ 0.27 & 0.14 & 0.52 & 0.47 & 0.44\\ \hline 0 & 0 & 0 & 0.27 & 0.45\\ 0 & 0 & 0 & 0.44 & 0.2 \end{array}\right],
			\left[\begin{array}{cc|c} 0.55 & 0.7 & 0.42\\ 0.74 & 0.03 & 0.69\\ 0.39 & 0.27 & 0.31\\\hline 0 & 0 & 0.91\\ 0 & 0 & 0.33 \end{array}\right],
			\left[\begin{array}{ccc|cc} 0.43 & 0.79 & 0.44 & 0.75 & 0.65\\ 0.38 & 0.18 & 0.64 & 0.27 & 0.16\\ \hline 0 & 0 & 0 & 0.67 & 0.11 \end{array}\right],
			\left[\begin{array}{cc|c} 1.0 & 0.15 & 0.2\\ 0.15 & 1.0 & 0.27\\\hline 0.2 & 0.27 & 1.0 \end{array}\right].
			\end{align*}
		\end{minipage}}\vspace{.3cm}\\	
		For the system matrices we generated random matrices and rounded them up to two digits. Furthermore, the appropriate blocks were changed to zero and the eigenvalues of the $A$ matrix were reduced to obtain a stable system. The output of Algorithm \ref{algoGrangerABCD} with input $(A,B,C,Q,[4,5])$ is given as	\\
		\resizebox{1 \linewidth}{!}
		{\begin{minipage}{1\linewidth}
		\begin{align*} 	
			&A_k,K,C_k = \\[.5em]
			&\left[\begin{array}{ccc|cc} -0.12 & 0.22 & 0 & -0.23 & -0.15\\ 0.13 & 0.78 & -0.1 & -0.9 & -0.44\\ 0.22 & -0.13 & 0.33 & 0.03 & 0.04\\ \hline 0 & 0 & 0 & -0.23 & -0.75\\ 0 & 0 & 0 & 0.02 & 0.7 \end{array}\right],
			\left[\begin{array}{cc|c} -0.08 & 0.49 & 0.3\\ -0.11 & -0.17 & 2.32\\ 0.24 & 0.06 & 0.04\\ \hline 0 & 0 & 2.35\\ 0 & 0 & -0.99 \end{array}\right],
			\left[\begin{array}{ccc|cc} -1.42 & -6.2 & 0.9 & 5.78 & -2.99\\ -0.63 & -4.18 & 0.39 & 3.76 & -1.84\\ \hline 0 & 0 & 0 & 0.08 & -0.46 \end{array}\right].
		\end{align*}
		\end{minipage}}\vspace{.3cm}\\
		We can see that the Kalman representation calculated by Algorithm \ref{algoGrangerABCD} is in block triangular form. Correspondingly, $[\y_1^T, \y_2^T]^T$ does not Granger cause $\y_3$. Note that for this example the noise covariance of the resulting system, thus the covariance of the innovation process of $\y$, equals $Q$. This is because in any case when the transfer matrix has stable inverse and the $D$ matrix is the identity the noise is in fact the innovation process of the output process. The output system of Algorithm \ref{algoGrangerY} with input $(\lbrace\Lambda_k^{\y}\rbrace_{k=0}^{M},[4,5])$, where $\Lambda_k^{\y}$s are the empirical covariances and $M$ is efficiently large, is the estimation of the system $(A_k,K,C_k,I)$ coming from Algorithm \ref{algoGrangerABCD}. Calculating the empirical covariances from $10^6$ simulations we obtained the following results: \\
		\resizebox{1 \linewidth}{!}
		{\begin{minipage}{1\linewidth}
				\begin{align*} 
			&\tilde{A}_k,\tilde{K},\tilde{C}_k  =\\[.5em]
			&\left[\begin{array}{ccc|cc} 0.79 & 0.34 & -0.02 & -0.91 & -0.46\\ 0.2 & -0.1 & -0.12 & -0.22 & -0.06\\ -0.15 & 0.06 & 0.38 & 0.38 & 0.18\\ \hline 0.07 & -0.01 & 0 & -0.31 & -0.78\\ 0.01 & 0 & 0 & 0.01 & 0.69 \end{array}\right],
			\left[\begin{array}{cc|c} -0.2 & 0.02 & 2.14\\ 0.05 & -0.5 & 0.82\\ -0.17 & -0.17 & -0.9\\ \hline 0.02 & 0.01 & 2.5\\ 0 & 0 & -1.0 \end{array}\right],
			\left[\begin{array}{ccc|cc} -6.14 & -1.51 & 1.5 & 5.88 & -2.84\\ -3.93 & -1.33 & 1.1 & 3.84 & -1.72\\ \hline 0 & 0 & 0 & 0.07 & -0.47 \end{array}\right].
		\end{align*}
		\end{minipage}}\vspace{.3cm}\\
		The algorithms are sensitive to the accuracy of the Markov parameters and also for disturbances of the system matrices. However, the statistical analysis of the resulting system form the algorithms in the presence of disturbance is beyond the scope of this paper. \\
			
	\textbf{Examples for algorithm 3 and 4:} Similarly, then in the previous example we define a state-space representation of a $\y=[\y_1^T, \y_2^T ,\y_3^T ]^T$ process $(A,B,C,I)$ with noise covariance $Q$ as given below. Note that the matrices in the two examples are the same apart from the zero blocks. \vspace{.3cm}\\
	\resizebox{1 \linewidth}{!}
	{\begin{minipage}{1\linewidth}
			\begin{align*} 	
			&A,B,C,Q :=\\[.5em]
			& \left[\begin{array}{c|cc|cc} 0.45 & 0 & 0 & 0.09 & 0.43\\ \hline 0 & 0.02 & 0.64 & 0.08 & 0.27\\ 0 & 0.14 & 0.52 & 0.47 & 0.44\\ \hline 0 & 0 & 0 & 0.27 & 0.45\\ 0 & 0 & 0 & 0.44 & 0.2 \end{array}\right],
			\left[\begin{array}{c|c|c} 0.55 & 0 & 0.42\\ \hline 0 & 0.03 & 0.69\\ 0 & 0.27 & 0.31\\ \hline 0 & 0 & 0.91\\ 0 & 0 & 0.33 \end{array}\right],
			\left[\begin{array}{c|cc|cc} 0.43 & 0 & 0 & 0.75 & 0.65\\ \hline 0 & 0.18 & 0.64 & 0.27 & 0.16\\ \hline 0 & 0 & 0 & 0.67 & 0.11 \end{array}\right],
			\left[\begin{array}{c|c|c} 1.0 & 0.15 & 0.2\\\hline 0.15 & 1.0 & 0.27\\ \hline 0.2 & 0.27 & 1.0 \end{array}\right].
			\end{align*}
		\end{minipage}}\vspace{.3cm}\\				
	The output of Algorithm \ref{algoCoordABCD} with input $(A,B,C,Q,[4,5])$ is given as\vspace{.3cm}\\
	\resizebox{1 \linewidth}{!}
	{\begin{minipage}{1\linewidth}
			\begin{align*} 	
			&A_k,K,C_k = \\[.5em]
			&\left[\begin{array}{c|cc|cc} 0.45 & 0 & 0 & -0.22 & -0.69\\\hline 0 & -0.12 & -0.1 & 0.02 & -0.1\\ 0 & -0.03 & 0.66 & 0.35 & 0.76\\\hline 0 & 0 & 0 & 0.57 & -0.22\\ 0 & 0 & 0 & -0.38 & -0.1 \end{array}\right],
			\left[\begin{array}{c|c|c} 0.33 & 0 & 1.14\\\hline 0 & -0.06 & 1.18\\ 0 & -0.11 & 17.37\\\hline 0 & 0 & 2.74\\ 0 & 0 & -17.11 \end{array}\right],
			\left[\begin{array}{c|cc|cc} 0.73 & 0 & 0 & 1.28 & -1.13\\\hline 0 & 0.17 & -1.76 & 0.15 & -1.78\\\hline 0 & 0 & 0 & 0.06 & -0.03 \end{array}\right].
			\end{align*}
		\end{minipage}}\vspace{.3cm}\\
	From the resulting system of Algorithm \ref{algoCoordABCD} we immediately can deduce that $[\y_1^T, \y_2^T]^T$ does not Granger cause $\y_3$ and it also holds that $\y_1$ and $\y_2$ does not Granger cause each other. In addition, the system $(A_k,K,C_k,I)$ is minimal. The output of Algorithm \ref{algoGrangerY} with input $(\Lambda_k,[4,5])$, where $\Lambda_k$ is an empirical covariance calculated from $10^6$ simulations, is given as\\
	\resizebox{1 \linewidth}{!}
	{\begin{minipage}{1\linewidth}
			\begin{align*} 
			&\tilde{A}_k,\tilde{K},\tilde{C}_k  =\\[.5em]
			&\left[\begin{array}{c|cc|cc} 0.44 & 0 & 0 & -0.15 & -0.71\\ \hline 0 & 0.53 & -0.38 & 0.25 & 0.74\\ 0 & -0.39 & 0.14 & -0.21 & -0.47\\\hline 0 & -0.06 & 0 & 0.58 & -0.27\\ -0.01 & -0.13 & 0 & -0.36 & -0.22 \end{array}\right],
			\left[\begin{array}{c|c|c} 0.33 & 0 & 1.39\\\hline 0 & -0.07 & 9.92\\ 0 & 0.12 & -9.01\\\hline 0 & 0.02 & 4.17\\ 0 & 0.03 & -13.1 \end{array}\right],
			\left[\begin{array}{c|cc|cc} 0.72 & 0 & 0 & 1.26 & -1.16\\\hline 0 & -1.55 & 1.0 & 0.15 & -1.86\\\hline 0 & 0 & 0 & 0.06 & -0.03 \end{array}\right].
			\end{align*}
	\end{minipage}}\vspace{.3cm}\\					
	As for the previous example, the noise covariance corresponding to the innovation process of $\y$ equals $Q$.
\section{Technical proofs} \label{Proofs}

\begin{proof}[Proof of Theorem \ref{Kalman22}]
	
	\textbf{If}
	Assuming $\y$ has a Kalman realization of the form
		\begin{align} \label{Km22abP}
		\begin{split}
		\begin{bmatrix} \x_1(t+1) \\ \x_2(t+1) \end{bmatrix} &=
		\underbrace{\begin{bmatrix}A_{1,1} & A_{1,2} \\ 0 & A_{2,2} \end{bmatrix}}_{A}
		\begin{bmatrix} \x_1(t) \\ \x_2(t) \end{bmatrix} + 
		\underbrace{\begin{bmatrix}K_{1,1} & K_{1,2} \\ 0 & K_{2,2} \end{bmatrix}}_{K} \underbrace{\begin{bmatrix} \e_1(t) \\ \e_2(t) \end{bmatrix}}_{\e}  \\
		\begin{bmatrix} \y_1(t) \\ \y_2(t) \end{bmatrix} &=
		\underbrace{\begin{bmatrix}C_{1,1} & C_{1,2} \\ 0 & C_{2,2} \end{bmatrix}}_{C}
		\begin{bmatrix} \x_1(t) \\ \x_2(t) \end{bmatrix} + \begin{bmatrix} \e_1(t) \\ \e_2(t) \end{bmatrix}. 
		\end{split}
		\end{align}
	where $(A_{2,2},K_{2,2},C_{2,2},I)$ is minimum phase we first show that $\e_2$ is in fact the innovation process of $\y_2$, \ie
	\begin{equation}\label{e2P}
	\e_2(t)=\y_2(t)-E_l[\y_2(t)|\mathcal{H}^{\y}_{t-}]=\y_2(t)-E_l[\y_2(t)|\mathcal{H}^{\y_2}_{t-}].
	\end{equation}
	The minimum phase assumption implies that the subsystem representing $\y_2$ has a causal stable inverse system such as 
	\begin{align*}
	\z(t+1)&=\tilde{A}\z(t) +\tilde{B}\y_2(t) \\
	\e_2(t)& =\tilde{C}\x_2(t) +\tilde{D}\y_2(t),
	\end{align*}
	from which we can deduce that $\e_2(t) \in \mathcal{H}^{\y_2}_{(t+1)-}$. Note that $A_{2,2}$ is stable, since $A$ is stable thus the equations \eqref{Km22abP} provide us $\x_2(t),\y_2(t)\in \mathcal{H}^{\e_2}_{(t+1)-}$. Consequently, we obtain that $\mathcal{H}^{\y_2}_{t-} = \mathcal{H}^{\e_2}_{t-}$. It follows then that $\x_2(t)\in \mathcal{H}^{\y_2}_{t-} \subset \mathcal{H}^{\y}_{t-}$. If we take now the projection of the output equation of $\y_2$ in \eqref{Km22abP} onto $\mathcal{H}^{\y_2}_{t-}$ and$ \mathcal{H}^{\y}_{t-}$ then in both cases the right hand side becomes $C_{2,2}\x_2(t)$. Therefore $E_l[\y_2(t)|\mathcal{H}^{\y_2}_{t-}]=E_l[\y_2(t)|\mathcal{H}^{\y}_{t-}]$ which gives \eqref{e2P}.	
	Next, we show that 
	\begin{equation} \label{k-stepP}
	E_l[\y_2(t+1) \mid \mathcal{H}^{\y}_{t-}]=E_l[\y_2(t+1) \mid \mathcal{H}^{\y_2}_{t-}] \hspace{.2cm} \forall t,k \in \mathbb{Z}, k \geq 0.
	\end{equation}
	From \eqref{Km22abP} we obtain that
	\[ \y_2(t+k)=\e_2(t+k)+ \sum_{j=1}^{\infty} C_{2,2}A^{j-1}_{2,2}K_{2,2}\e_2(t+k-j). \]
	Notice that $\e_2(t+j)$ is orthogonal to $\mathcal{H}_{t-}^{\y}$ for $j \ge 0$ since
	$\e_2(t+j)=\y_2(t+j)-E_l[\y_2(t+j) \mid  \mathcal{H}_{(t+j)-}^{\y}]$. Also, $\e_2(t-j-1)$ belongs to $\mathcal{H}_{t-}^{\y}$.
	Consequently, for any $k \geq 0$
	$$E_l[\y_2(t+k)\mid \mathcal{H}_{t-}^{\y}]=\sum_{j=1}^{\infty}C_{2,2}A^{k+j-1}_{2,2}K_{2,2}\e_2(t-j).$$
	On the other hand $\e_2(t+j)$ is orthogonal to $\mathcal{H}_{t-}^{\y_2}$ since $\e_2(t+j)$ is orthogonal to $\mathcal{H}_{t-}^{\y} \supseteq \mathcal{H}_{t-}^{\y_2}$.	Moreover, $\e_2(t-j-1)$ belongs to $\mathcal{H}_{t-}^{\y_2}$ because $\e_2(t)$ was shown to be the innovation process of $\y_2(t)$. Therefore, the equation \eqref{k-stepP} for the $k$-step prediction follows.
	 
	\textbf{Only if}
	To begin with, based on realization theory since $\y$ is ZMSIR there exists a linear deterministic realization
	$(A,B,C,R) \in \mathbb{R}^{n \times n} \times \mathbb{R}^{p \times n} \times \mathbb{R}^{n \times p} \times \mathbb{R}^{p \times p}$, for $p:=\dim(\y(t))$ and some $n \ge 1$ such that $(A,C)$ is observable, $(A,B)$ is controllable, $A$ stable and 
	\[ \forall k \ge 1: \Lambda_k^{\y}=E[\y(t+k)\y^T(t)]=CA^{k-1}B, \hspace{.3cm} \Lambda_0^{\y}=R. \]
	Consider the decomposition $C=\begin{bmatrix} \widehat{C}_1 \\ \widehat{C}_2 \end{bmatrix}$, $I=\begin{bmatrix} I_1 \\ I_2 \end{bmatrix}$ where $(\widehat{C}_2, I_2) \in \mathbb{R}^{p_2 \times n}\times \mathbb{R}^{p_2 \times p_2}$, $p_2:=\dim(\y_2(t))$. First we discuss the case when  $(\widehat{C}_2,A)$ is a non-observable pair. We discuss the observable case at the end of the proof because it can be carried out the same way using simplifications.
	If $(\widehat{C}_2,A)$ is a non-observable pair then by applying a suitable basis transformation we can 
	assume that $(A,B,\widehat{C}_2)$ is of the following form:
	\begin{equation}\label{triang}
	A = \begin{bmatrix} A_{1,1} & A_{1,2} \\ 0 & A_{2,2} \end{bmatrix},\quad
	B = \begin{bmatrix} B_{1,1} & B_{1,2} \\ B_{2,1} & B_{2,2} \end{bmatrix}, \quad
	\widehat{C}_2 = \begin{bmatrix}  0 & C_{2,2}  \end{bmatrix},
	\end{equation}
	where $A_{2,2} \in \mathbb{R}^{n_2 \times n_2}, B_{2,2} \in \mathbb{R}^{n_2 \times p_2}, C_{2,2} \in \mathbb{R}^{p_2 \times n_2}$ for some $n_2 \ge 1$ such that $(C_{2,2},A_{2,2})$ is an observable pair. 		
	Note that stability of $A$ implies stability of $A_{2,2}$. 
	Define for $i=1,2$ the processes
	\[ Y_i(t)=\begin{bmatrix} \y_i(t) \\ \y_i(t+1) \\ \vdots \\ \y_i(t+n-1) \end{bmatrix}, \quad  Y(t)=\begin{bmatrix} \y(t) \\ \y(t+1) \\ \vdots \\ \y(t+n-1) \end{bmatrix}.\] 
	and the observability matrices
	\[ O_1 = \begin{bmatrix} C_{1,1} \\ C_{1,1}A_{1,1} \\ \vdots \\ C_{1,1}A_{1,1}^{n-1} \end{bmatrix}, \quad O_2 = \begin{bmatrix} \widehat{C}_2 \\ \widehat{C}_2A \\ \vdots \\ \widehat{C}_2A^{n-1} \end{bmatrix}, \quad O = \begin{bmatrix} C \\ CA \\ \vdots \\ CA^{n-1} \end{bmatrix}.\]
	Take a permutation matrix $P \in \mathbb{R}^{pn \times pn}$, which splits the rows corresponding to $\y_1$ and $\y_2$, such that
	\[ PO=\begin{bmatrix} O_{1,1} & O_{1,2} \\ 0 & O_2 \end{bmatrix},  \quad PY(t) = \begin{bmatrix} Y_1(t) \\ Y_2(t) \end{bmatrix}. \]
	where $\begin{bmatrix} O_{1,1} & O_{1,2} \end{bmatrix}=O_1$. Since from observability of $(C,A)$ and $(C_{2,2},A_{2,2})$ it follows that $O_2$ and $O$ are full column rank matrices and thus we can define
	\begin{equation}\label{state}
	\x(t) := (PO)^{+} E_l[\begin{bmatrix} Y_1(t)  \\ Y_2(t) \end{bmatrix} \mid \mathcal{H}_{t-}^{\y}],
	\end{equation}
	where
	\[ (PO)^{+}=\begin{bmatrix} O_{1,1}^{+} & -O_{1,1}^{+}O_{1,2}O_2^{+} \\ 0 & O_{2}^{+} \end{bmatrix}.\]
	Partition $\x(t)$ into
	\[ \x(t)=\begin{bmatrix} \x_1(t) \\ \x_2(t) \end{bmatrix}, \]
	where $\x_2(t) \in \mathbb{R}^{n_2}$.  It is easy to see that 
	\[ \x_2(t) = O_2^{+} E_l[Y_2(t)\mid \mathcal{H}_{t-}^{\y}], \]
	and hence from the assumption that $\y_1$ is Granger noncausal for $\y_2$
	\[ \x_2(t) = O_2^{+} E_l[Y_2(t)\mid \mathcal{H}_{t-}^{\y}]=O_2^{+} E_l[Y_2(t) \mid \mathcal{H}^{\y_2}_{t-}]. \]
	From  the definition of $\Lambda_k^{\y}$, $k \ge 1$ it follows that
	\begin{equation} \label{Covy}
	\begin{split} 
	& E[\begin{bmatrix} Y_1(t) \\ Y_2(t) \end{bmatrix}\y^T(t-k)] = P E[Y(t)\y^T(t-k)]=P\begin{bmatrix} \Lambda_{k} \\ \vdots \\ \Lambda_{k+n-1} \end{bmatrix}= PO A^{k-1}B,
	\end{split}
	\end{equation}
	and therefore, by using $E[\begin{bmatrix} Y_1(t) \\ Y_2(t) \end{bmatrix}\y^T(t-k)] = E[E_l[\begin{bmatrix} Y_1(t)  \\ Y_2(t) \end{bmatrix} \mid \mathcal{H}_{t-}^{\y}]\y^T(t-k)]$ we can conclude that
	\begin{equation} \label{granger:eq2}
	E[\x(t)\y^T(t-k)] = A^{k-1}B. 
	\end{equation}
	From \eqref{Covy} it follows that
	$$(PO)^{+}E_l[\begin{bmatrix} Y_1(t+1) \\ Y_2(t+1)\end{bmatrix}| \mathcal{H}_{t-}^{\y}]=A(PO)^{+}E_l[\begin{bmatrix} Y_1(t) \\ Y_2(t)\end{bmatrix}| \mathcal{H}_{t-}^{\y}]$$
	which implies 
	\begin{equation} \label{granger:onlyif_A}
	E_l[\x(t+1)|\mathcal{H}_{t-}^{\y}]=A\x(t).
	\end{equation}
	Define now $\e_i(t)=\y_i(t)-E_l[\y_i(t) \mid \mathcal{H}^{\y}_{t-}]$ for $i=1,2$ and join them as
	\[ \e(t)=\begin{bmatrix} \e_1(t) \\ \e_2(t) \end{bmatrix} = \y(t) - E_l[\y(t) \mid \mathcal{H}^{\y}_{t-}], \]
	where $\e_2(t)$ is orthogonal to $\mathcal{H}^{\y}_{t-}$ and hence also to $\mathcal{H}^{\y_2}_{t-}\supseteq\mathcal{H}^{\y}_{t-}$. Note that
	\[ \mathcal{H}_{(t+1)-}^{\y}=\mathcal{H}^{\y}_{t-} \oplus \mathcal{H}^{\e}_{t} ,\hspace{.15cm}
	\mathcal{H}_{(t+1)-}^{\y_2}=\mathcal{H}^{\y_2}_{t-} \oplus  \mathcal{H}^{\e_2}_{t}.\] 
	The components of $\x(t+1)$ belong to $\mathcal{H}_{(t+1)-}^{\y}$ and also, from the assumption that $\y_1$ does is Granger noncausal for $\y_2$, we know that the components of $\x_2(t+1)=O_2^{+} E_l[Y_2(t+1)\mid H_{(t+1)-}^{\y}]=O_2^{+} E_l[Y_2(t+1) \mid \mathcal{H}^{\y_2}_{(t+1)-}]$ belong to $\mathcal{H}^{\y_2}_{(t+1)-}$. It then follows that
	\[ \begin{split}
	& \x_1(t+1)=E_l[\x_1(t+1) \mid \mathcal{H}^{\y}_{t-}] + E_l[\x_1(t+1) \mid \mathcal{H}^{\e}_{t}]  \\
	& \x_2(t+1)=E_l[\x_2(t+1) \mid \mathcal{H}^{\y_2}_{t-}] + E_l[\x_2(t+1) \mid \mathcal{H}^{\e_2}_{t}]. \\
	\end{split} \]
	For a suitable $K=\begin{bmatrix}K_{1,1} & K_{1,2} \\ 0 & K_{2,2} \end{bmatrix}$ matrix, which is indeed the Kalman gain, we have that
	\[ \begin{split}
	E_l[\x_2(t+1) \mid \mathcal{H}^{\e_2}_{t}]=K_{2,2}\e_2(t), \quad E_l[\x_1(t+1) \mid \mathcal{H}^{\e}_{t}]=K_{1,1}\e_1(t)+K_{1,2}\e_2(t). 
	\end{split} \]
	Combining it with \eqref{granger:onlyif_A} we obtain the state equation below.
	\begin{align} \label{granger:eq7}
	\begin{split}
	\begin{bmatrix} \x_1(t+1) \\ \x_2(t+1) \end{bmatrix} &=
	\begin{bmatrix}A_{1,1} & A_{1,2} \\ 0 & A_{2,2} \end{bmatrix}
	\begin{bmatrix} \x_1(t) \\ \x_2(t) \end{bmatrix} + 
	\begin{bmatrix}K_{1,1} & K_{1,2} \\ 0 & K_{2,2} \end{bmatrix} \begin{bmatrix} \e_1(t) \\ \e_2(t) \end{bmatrix} 
	\end{split}
	\end{align}
	Finally, the components of $\x(t)$ belong to $\mathcal{H}_{t-}^{\y}$ and from \eqref{granger:eq2} it follows that
	\[ CE[\x(t)\y^T(t-k)] = CA^{k-1}B=\Lambda_k = E[\y(t)\y^T(t-k)], \]
	i.e. $\y(t)-C\x(t)$ is orthogonal to $\mathcal{H}^{\y}_{t-}$. Hence, $\y(t)=C\x(t)+\e(t)$.
	Since $C$ has block triangular form the output equation can be written as
	\begin{align} \label{granger:eq9}
	\begin{split}
	\begin{bmatrix} \y_1(t) \\ \y_2(t) \end{bmatrix} &=
	\begin{bmatrix}C_{1,1} & C_{1,2} \\ 0 & C_{2,2} \end{bmatrix}
	\begin{bmatrix} \x_1(t) \\ \x_2(t) \end{bmatrix} + \begin{bmatrix} \e_1(t) \\ \e_2(t) \end{bmatrix}. 
	\end{split}
	\end{align}
	The equations \eqref{granger:eq7} and \eqref{granger:eq9} yield a state-space representation as in \eqref{Km22ab}. Since $(A_{2,2},K_{2,2},C_{2,2},I)$ gives an innovation representation for $\y_2$ it has a stable causal inverse system. Therefore the minimum phase condition follows. At last, minimality of the constructed system comes from realization theory. \\
	Now we discuss the case when $(\widehat{C}_2,A)$ is an observable pair. Defining $O_1$ and $O_2$ as the observability matrix of $(\widehat{C}_2,A)$ and $(\widehat{C}_1,A)$, respectively, the permuted observability matrix is $PO=\begin{bmatrix}O_1 \\O_2 \end{bmatrix}$. Considering that now $(PO)^{+}=\begin{bmatrix}0 & O_2^{+} \end{bmatrix}$ the state process defined in \eqref{state} becomes
	$$\x(t)=\begin{bmatrix}0 & O_2^{+}\end{bmatrix}E_l[\begin{bmatrix} Y_1(t)\\Y_2(t)	\end{bmatrix}|\mathcal{H}^{\y}_{t-}]=O_2^{+}E_l[Y_2(t)|\mathcal{H}^{\y}_{t-}]=O_2^{+}E_l[Y_2(t)|\mathcal{H}^{\y_2}_{t-}]=\x_2(t),$$
	where we used that $\y_1(t)$ is Granger noncausal for $\y_2$. We can see that $\dim(\x_1(t))=0$, but the rest of the proof, particularly the calculation of \eqref{granger:eq7} and \eqref{granger:eq9}, remains the same. Consequently, the state-space representation of $\y$ when $(\widehat{C}_2,A)$ is observable is given by
	\begin{align*}
	\x_2(t) &= A_{2,2}\x_2(t) + \begin{bmatrix}0 & K_{2,2}\end{bmatrix}\begin{bmatrix}\e_1(t) \\\e_2(t)\end{bmatrix}\\
	\y(t) & = C_{2,2}\x_2(t) + \begin{bmatrix}\e_1(t) \\\e_2(t)\end{bmatrix},
	\end{align*}
	where $A_{2,2}=A$ and $C_{2,2}=C$.
\end{proof}

\begin{proof}[Proof of Lemma \ref{Algo12}]
	In the proof of Theorem \ref{Kalman22} we constructed the representation \eqref{Km22ab} by using an $(A,B,\bar{C})$ factorization of the Markov parameters $\Lambda_k^{\y}$ of a process $\y$ such that $BA^k\bar{C}=\Lambda_k^{\y}$. By realization theory we know that defining $\bar{C}$ as $CPA^T+DB^T$ for a state-space representation $(A,B,C,D)$ of $\y$ where $P$ is the state covariance the equation $BA^k\bar{C}=\Lambda_k^{\y}$ is satisfied for all $k\geq 0$. Notice that in Step 1 of Algorithm \ref{algoGrangerABCD} we defined $\bar{C}$ exactly this way. Also, in Step 1 and Step 2 of Algorithm \ref{algoGrangerY} we choose $(A,C,\bar{C})$ to satisfy $BA^k\bar{C}=\Lambda_k^{\y}$, $k\geq 0$. The succeeding steps of both algorithms are designed according to the proof of Theorem \ref{Kalman22}. Consequently, the resulting system is in block triangular form if and only if the Granger non-causality condition holds.
\end{proof}

\begin{proof}[Proof of Theorem \ref{GC22}]
	\textbf{If} The existence of system \eqref{GC22ab} follows from the proof of Theorem \ref{Kalman22} where we created the system such that it satisfies \eqref{GC22i}-\eqref{GC22iii} in Theorem \ref{GC22}. Condition $\eqref{GC22i}$ holds since the resulting system is in innovation form. Condition $\eqref{GC22ii}$ and $\eqref{GC22iii}$ come from fundamental result for Kalman realization.
	
	\textbf{Only if} Assuming that $\y$ has a realization of the form \eqref{GC22ab} and that (\ref{GC22i})-(\ref{GC22iii}) hold we prove that 
	\begin{equation}\label{GC22P}
	E_l[\y_2(t+1) \mid \mathcal{H}^{\y}_{t-}]=E_l[\y_2(t+1) \mid \mathcal{H}^{\y_2}_{t-}] \hspace{.2cm} \forall t \in \mathbb{Z},
	\end{equation}
	i.e. that $\e_2$ is in fact the innovation process of $\y_2$. The rest of the proof, to see \eqref{GC22P} for $k$-step predictions, is the same as in the proof of Theorem \ref{Kalman22}.
	Let $Q_{\hat{\x}}$ be the minimal positive definite solution of \eqref{Km22gain} which  exists by \cite[Theorem 6.9.3]{LinPic15}. From stochastic realization theory it follows that if we choose $\hat{K}=(\bar{C}_{2,2} - A_{2,2}Q_{\hat{\x}}C_{2,2}^T)\left(\Lambda_0^{\y_2}-CQ_{\hat{\x}}C^T\right)^{-1}$, then there is a realization of $\y_2$ in the form of
	\[\begin{split}
	\hat{\x}(t+1) &=A_{2,2}\hat{\x}(t)+\hat{K}\hat{\e}(t) \\
	\y_2(t) & =C_{2,2}\hat{\x}(t)+\hat{\e}(t)
	\end{split}	\]
	with $\hat{\e}(t)=\y_2(t)-E_l[\y_2(t) \mid \{ \y_2(t-l)\}_{l=1}^{\infty}]$ being the innovation of $\y_2(t)$. From assumption \eqref{GC22ii} it follows that $\hat{P}=Q_{\x_2}$. In particular, $E[\hat{\e}(t)\hat{\e}^T(t)]=E[\e_2(t)\e_2^T(t)]$ which means that the projection errors of $\y_2$ onto $\mathcal{H}_{t-}^{\y_2}$ and $\mathcal{H}_{t-}^{\y}$ are equal. Since $\mathcal{H}_{t-}^{\y_2} \subset \mathcal{H}_{t-}^{\y}$ it already verifies \eqref{GC22P}. \\
\end{proof}

For the proof of Theorem \ref{GCnn} we need the following two lemmas:

\begin{Lemma} \label{GC123}
	Assume that $\y=[\y_1^T,\y_2^T,\y_3^T]^T$ is a ZMSIR process. Then $\y_1$ and $\y_2$ does not Granger cause $\y_3$ if and only if $[\y_1^T,\y_2^T]^T$ does not Granger cause $\y_3$.
\end{Lemma}

\begin{Lemma} \label{Lem}
		If $\y_i$ is Granger noncausal for $\y_n$ for all $i \in \lbrace1,2,\ldots, n-1 \rbrace$ than there is a linear stochastic state-space representation of $\y$ in the form of		
		\begin{multicols}{2}
			\resizebox{1.88\linewidth}{!}
			{\hskip -9 mm
				\begin{minipage}{\linewidth}
					\begin{align*} 
					\left[ \! \begin{array}{c} \x_1(t+1) \\ \x_2(t+1) \\ \vdots \\ \x_{n-1}(t+1) \\ \hline \x_n(t+1) \end{array}\!\right] \! &\mathord{=} \!
					\underbrace{\left[ \! \begin{array}{cccc|c} A_{1,1} & 0 & \dots & 0 & A_{1,n} \\ 0 & A_{2,2} & \dots & 0 & A_{2,n} \\ \vdots & \vdots & \ddots & \vdots & \vdots \\ 0 & 0 & \dots & A_{n-1,n-1} & A_{n-1,n} \\\hline 0 & 0 & \dots & 0 & A_{n,n}\end{array}\!\right]}_A
					\left[ \! \begin{array}{c} \x_1(t) \\ \x_2(t) \\ \vdots \\ \x_{n-1}(t) \\ \hline \x_n(t) \end{array}\!\right] \!	\mathord{+} \!
					\underbrace{ \left[\! \begin{array}{cccc|c} K_{1,1} & 0 & \dots & 0 & K_{1,n} \\ 0 & K_{2,2} & \dots & 0 & K_{2,n} \\ \vdots & \vdots & \ddots & \vdots & \vdots \\ 0 & 0 & \dots & K_{n-1,n-1} & K_{n-1,n} \\	\hline 0 & 0 & \dots & 0 & K_{n,n}\end{array} \! \right] \!}_K
					\left[ \! \begin{array}{c}  \e_1(t) \\ \e_2(t) \\ \vdots \\ \e_{n-1}(t) \\ \hline \e_n(t) \end{array}\!\right]  \\
					\left[ \! \begin{array}{c}  \y_1(t) \\ \y_2(t) \\ \vdots \\ \y_{n-1}(t) \\ \hline \y_n(t) \end{array}\!\right] &\mathord{=}
					\underbrace{\left[\!\begin{array}{cccc|c} C_{1,1} & 0 & \dots & 0 & C_{1,n} \\ 0 & C_{2,2} & \dots & 0 & C_{2,n} \\ \vdots & \vdots & \ddots & \vdots & \vdots \\ 0 & 0 & \dots & C_{n-1,n-1} & C_{n-1,n} \\	\hline 0 & 0 & \dots & 0 & C_{n,n}\end{array}\! \right]}_C \!
					\left[ \! \begin{array}{c}  \x_1(t) \\ \x_2(t) \\ \vdots \\ \x_{n-1}(t) \\ \hline \x_n(t)\end{array}\!\right] \!+\! 
					\left[ \! \begin{array}{c}  \e_1(t) \\ \e_2(t) \\ \vdots \\ \e_{n-1}(t) \\ \hline \e_n(t) \end{array}\!\right] ,
					\end{align*}
				\end{minipage}
			}\break
			\vskip 11mm
			\begin{equation}
			\label{repr1n_Lem}
			\end{equation}
		\end{multicols}	
		such that the following holds:
		\begin{enumerate}[(i)]
			\item \label{Lemi} $\e_i(t)=\y_i(t)-E_l[\y_i(t)| \mathcal{H}_{t-}^{\y_i,\y_n}], \hspace{.2cm}  i=1,2,\dots, (n-1);$
			\item \label{Lemii} $\e_n(t)=\y_n(t)-E_l[\y_n(t)|\mathcal{H}^{\y}_{t-}];$
			\item \label{Lemiv} the matrix $Q_{\x_n}:=E[\x_n(t)\x_n^T(t)]$ is the minimal positive definite solution of
			\begin{equation*}
			\Sigma=A_{n,n}\Sigma A_{n,n}^T+(G_{n,n}-A_{n,n}\Sigma C_{n,n}^T)( \Lambda_0^{\y_n}- C_{n,n}\Sigma C_{n,n}^T)^{-1}(G_{n,n}-A_{n,n}\Sigma C_{n,n}^T)^T;
			\end{equation*}
			\item \label{Lemiii} the subsystem representing $\y_n$ is minimal;
			\item \label{Lemv} the state-space is constructible;
			\item \label{Lemvi} $(A,C)$ is observable.
		\end{enumerate}
\end{Lemma}
\begin{proof}[proof of Lemma \ref{GC123}]
	\textbf{If} By definition, the joint process $[\y_1^T,\y_2^T]^T$ does not Granger cause $\y_3$ if for all $t,k\in \mathbb{Z},k\geq 0$
	$$E_l[\y_3(t+k)|\mathcal{H}^{\y_3}_{t-}]=E_l[\y_3(t+k)|\mathcal{H}^{\y}_{t-}].$$
	Projecting both sides onto $\mathcal{H}^{\y_1,\y_3}_{t-}$ and to $\mathcal{H}^{\y_2,\y_3}_{t-}$ we obtain that
	$$E_l[\y_3(t+k)|\mathcal{H}^{\y_3}_{t-}]=E_l[\y_3(t+k)|\mathcal{H}^{\y_1,\y_3}_{t-}], \hspace{.3cm}E_l[\y_3(t+k)|\mathcal{H}^{\y_3}_{t-}]=E_l[\y_3(t+k)|\mathcal{H}^{\y_2,\y_3}_{t-}],$$ 
	which implies that $\y_1$ and $\y_2$ does not Granger cause $\y_3$.\\
	
	\textbf{Only if} For a fixed $t\in \mathbb{Z}$ define the processes $\alpha_{k} := \y_3(t + k) - E_l[\y_3(t + k)|\mathcal{H}_{t-}^{\y_3} ]$, $k\in\mathbb{Z}$, $k\geq 0$. It then follows that $\alpha_{t+k} \perp \mathcal{H}_{t-}^{\y_3}$ and from the conditions we obtain that 
	$\alpha_{t+k} \perp  \mathcal{H}_{t-}^{\y_1,\y_2}$ and $\mathcal{H}_{t-}^{\y_2,\y_3}$. Therefore, $\alpha_{t+k}$ is orthogonal to $\mathcal{H}_{t-}^{\y_1,\y_3}+\mathcal{H}_{t-}^{\y_2,\y_3}=\mathcal{H}_{t-}^{\y}$, the Hilbert space generated by $\y$. Projecting $\alpha_{t+k}$ onto $\mathcal{H}_{t-}^{\y}$ we obtain that
	$$ E_l[\y_3 (t + k) | \mathcal{H}_{t-}^{\y} ] = E_l [\y_3 (t + k) | \mathcal{H}_{t-}^{\y_3} ],$$ which by definition is that  $[\y_1,\y_2]$ does not Granger cause $\y_3$.
\end{proof}

\begin{proof}[proof of Lemma \ref{Lem}]
	Applying Theorem \ref{GC22} for $[\y_i^T,\y_n^T]^T$ we obtain minimal innovation representations such as
	\begin{align} \label{reprin}
		\begin{split}
		\begin{bmatrix} \tilde{\x}_i(t+1) \\ \tensor[_i]{\tilde{\x}}{_n}(t+1) \end{bmatrix} &=
		\begin{bmatrix}\tilde{A}_{i,i} & \tilde{A}_{i,n} \\  0 &  \tensor[_i]{\tilde{A}}{_{n,n}} \end{bmatrix}
		\begin{bmatrix} \tilde{\x}_i(t) \\ \tensor[_i]{\tilde{\x}}{_n}(t) \end{bmatrix} + 
		\begin{bmatrix}\tilde{K}_{i,i} & \tilde{K}_{i,n} \\  0 &  \tensor[_i]{\tilde{K}}{_{n,n}}  \end{bmatrix}
		\begin{bmatrix} \tilde{\e}_i(t) \\ \tensor[_i]{\tilde{\e}}{_n}(t) \end{bmatrix}  \\
		\begin{bmatrix} \y_i(t) \\ \y_n(t) \end{bmatrix} &=
		\begin{bmatrix}\tilde{C}_{i,i} & \tilde{C}_{i,n} \\  0 &  \tensor[_i]{\tilde{C}}{_{n,n}} \end{bmatrix}
		\begin{bmatrix} \tilde{\x}_i(t) \\ \tensor[_i]{\tilde{\x}}{_n}(t) \end{bmatrix} + \begin{bmatrix} \tilde{\e}_i(t) \\ \tensor[_i]{\tilde{\e}}{_n}(t) \end{bmatrix}, 
		\end{split}
	\end{align}
		where the following holds:
	\begin{enumerate}[(i)]
		\item \label{LemPri} $\tilde{\e}_i(t)=\y_i(t)-E_l[\y_i(t)| \mathcal{H}_{t-}^{\y_i,\y_n}]$;
		\item \label{LemPrii} $\tensor[_i]{\tilde{\e}}{_n}(t)=\y_n(t)-E_l[\y_n(t)|\mathcal{H}_{t-}^{\y_i,\y_n}]=\y_n(t)-E_l[\y_n(t)|\mathcal{H}^{\y}_{t-}]$;
		\item \label{LemPriii} the matrix $\tensor[_i]{Q}{_{\x_n}}:=E[\tensor[_i]{\tilde{\x}}{_{n}}(t)\tensor[_i]{\tilde{\x}}{_{n}}^T(t)]$ is the minimal positive definite solution of
		\begin{equation*}
			\Sigma=\tensor[_i]{\tilde{A}}{_{n,n}}\Sigma\tensor[_i]{\tilde{A}}{_{n,n}}^T+(\tensor[_i]{G}{_{n,n}}-\tensor[_i]{\tilde{A}}{_{n,n}}\Sigma \tensor[_i]{\tilde{C}}{_{n,n}}^T)( \Lambda_0^{\y_n}- \tensor[_i]{\tilde{C}}{_{n,n}}\Sigma \tensor[_i]{\tilde{C}}{_{n,n}}^T)^{-1}(\tensor[_i]{G}{_{n,n}}-\tensor[_i]{\tilde{A}}{_{n,n}}\Sigma \tensor[_i]{\tilde{C}}{_{n,n}}^T)^T;
		\end{equation*}
		\item \label{LemPriv} the subsystem representing $\y_n$ is minimal;
		\item \label{LemPrv} $[\tilde{\x}_i^T(t),\tensor[_i]{\tilde{\x}}{_n}^T(t)]^T \in \mathcal{H}_{t-}^{\y_i,\y_n}$. 
	\end{enumerate}
	Note that the second equation in \eqref{LemPrii} is the consequence of Lemma \ref{GC123} but, aside from that, \eqref{LemPri}-\eqref{LemPrv} come from the construction of the subsystems according to the proof of Theorem \ref{GC22}. Since the matrices $(\tensor[_i]{\tilde{A}}{_{n,n}},\tensor[_i]{\tilde{K}}{_{n,n}},\tensor[_i]{\tilde{C}}{_{n,n}},I)$  in \eqref{reprin} define minimal realizations for $\y_n$ in innovation form, they are isomorphic and thus there exist nonsingular $T_i$ matrices such that $\tensor[_i]{\tilde{\x}}{_n}=T_i \hspace{.05cm}\tensor[_1]{\tilde{\x}}{_n}$. Defining $T_1$ as the identity matrix and $T_i$ by $\tensor[_1]{\tilde{\x}}{_n}=T_i \hspace{.05cm}\tensor[_1]{\tilde{\x}}{_n}$ we can merge the representations \eqref{reprin} into the form (\ref{repr1n_Lem}) with the matching below.\\
	\resizebox{1\linewidth}{!}
	{
	\begin{minipage}{\linewidth}
		\begin{align*}
			A:=\left[ \! \begin{array}{cccc|c} \tilde{A}_{1,1} & 0 & \dots & 0 & \tilde{A}_{1,n}T_1 \\ 0 & \tilde{A}_{2,2} & \dots & 0 & \tilde{A}_{2,n}T_2 \\ \vdots & \vdots & \ddots & \vdots & \vdots \\ 0 & 0 & \dots & \tilde{A}_{n-1,n-1} & \tilde{A}_{n-1,n}T_{n-1} \\\hline 0 & 0 & \dots & 0 & \tensor[_1]{\tilde{A}}{_{n,n}}\end{array}\!\right]; &&
			K:=\left[ \! \begin{array}{cccc|c} \tilde{K}_{1,1} & 0 & \dots & 0 & \tilde{K}_{1,n} \\ 0 & \tilde{K}_{2,2} & \dots & 0 & \tilde{K}_{2,n} \\ \vdots & \vdots & \ddots & \vdots & \vdots \\ 0 & 0 & \dots & \tilde{K}_{n-1,n-1} & \tilde{K}_{n-1,n} \\ \hline 0 & 0 & \dots & 0 & \tensor[_1]{\tilde{K}}{_{n,n}}\end{array}\!\right]; \\[.3cm]				
			C:=\left[ \! \begin{array}{cccc|c} \tilde{C}_{1,1} & 0 & \dots & 0 & \tilde{C}_{1,n}T_1 \\ 0 & \tilde{C}_{2,2} & \dots & 0 & \tilde{C}_{2,n}T_2 \\ \vdots & \vdots & \ddots & \vdots & \vdots \\ 0 & 0 & \dots & \tilde{C}_{n-1,n-1} & \tilde{C}_{n-1,n}T_{n-1} \\\hline 0 & 0 & \dots & 0 & \tensor[_1]{\tilde{C}}{_{n,n}}\end{array}\!\right]; &&
			\begin{matrix}
			\x(t):=\begin{bmatrix} \tilde{\x}_1^T(t) & \dots & \tilde{\x}_{n-1}^T(t) & \tensor[_1]{\tilde{\x}}{_n}^T\end{bmatrix}^T;\hspace{.5cm} \\
			\\
			\e(t):=\begin{bmatrix} \tilde{\e}_1^T(t) & \dots & \tilde{\e}_{n-1}^T(t) & \tensor[_1]{\tilde{\e}}{_n}^T\end{bmatrix}^T.\hspace{.5cm} 
			\end{matrix}
		\end{align*}
	\end{minipage}
	}\break
		
	We mention that since $\x(t) \in \mathcal{H}_{t-}^{\y}$ is by definition the constructibility of the state-space \eqref{Lemv} follows. This construction indicates that the conditions \eqref{Lemi}-\eqref{Lemv} are satisfied thus it only remains to prove \eqref{Lemvi}. To see that the pair $(C,A)$ is observable consider that observability is equivalent with the full rank property of the matrix $\begin{bmatrix} C \\ I-\lambda A \end{bmatrix}$ for all $\lambda \in \mathbb{C}$. From the minimality of (\ref{reprin}) we have that $(\tilde{C}_{i,i},\tilde{A}_{i,i})$ are observable pairs so that $\begin{bmatrix} \tilde{C}_{i,i} \\ I-\lambda \tilde{A}_{i,i} \end{bmatrix}$ has full rank for all $\lambda \in \mathbb{C}$. A transformation of $\begin{bmatrix} C \\ I-\lambda A \end{bmatrix}$ into an upper block triangular form such that the block diagonal submatrices are $\begin{bmatrix} \tilde{C}_{i,i} \\ I-\lambda \tilde{A}_{i,i} \end{bmatrix}$ gives the observability of $(A,C)$ and completes the proof.
\end{proof}

\begin{proof}[Proof of Theorem \ref{GCnn}]
	\textbf{Only if} Supposing that the representation \eqref{Cor_nn} exists with the properties \eqref{Thi}-\eqref{Thii}-\eqref{Thiii}, we first show that (\ref{GC1ni}*) and (\ref{GC1nii}*) hold and then we verify that additionally the minimality of such a representation implies \eqref{Thiv}. Applying Theorem \ref{GC22}, a minimal stochastic representation in the form \eqref{Cor_nn} with the properties \eqref{Thi}-\eqref{Thiii} implies that $[\y_1^T,\y_2^T,\ldots,\y_{n-1}^T]$ is Granger noncausal for $\y_n$ and therefore condition (\ref{GC1ni}*) holds. For condition (\ref{GC1nii}*) we consider the equations $\sum_{j=0}^{\infty}$
	\begin{align}\label{sseq} 
		\begin{split} 
		\x_n(t+1) & = \sum_{j=0}^{\infty} A_{n,n}^{j}K_{n,n}\e_n(t-j) \\
		\x_i(t+1) & = \sum_{l=0}^{\infty}A_{i,i}^{l}\left(K_{i,i}\e_i(t-l)+K_{i,n}\e_n(t-l)\right)+A_{i,n}\sum_{j=0}^{\infty}A_{n,n}^{j}K_{n,n}\e_n(t-j-1)
		\end{split}
	\end{align}
	for $i \in \lbrace 1,2,\dots , n-1 \rbrace$. Throughout the proof $i$ will be an element of the set $\lbrace1,2,\dots, n-1\rbrace$. From \eqref{sseq} and property \eqref{GC1ni} it follows that $\mathcal{H}_{t}^{\x_i} \subseteq \mathcal{H}_{t-}^{\e_i,\e_n}=\mathcal{H}_{t-}^{\y_i,\y_n}$ and also that $\mathcal{H}_{t}^{\x_n} \subseteq \mathcal{H}_{t-}^{\e_n}=\mathcal{H}_{t-}^{\y_n}$. Observe that for $k \in \mathbb{Z}^{+}$
	\begin{align} \label{outeq}
		\begin{split} 
		\y_n(t+k) & = C_{n,n}A_{n,n}^{k}\x_n(t)+\z_n, \hspace{.2cm} \z_n \in \operatorname{span} \lbrace \e_n(t+j)\rbrace_{j=0}^k\\[.5em]
		\y_i(t+k) & = C_{i,i}A_{i,i}^{k}\x_i(t)+ C_{i,n}A_{n,n}^{k-1}\x_n(t)+\z_{i,n}, \hspace{.2cm} \z_{i,n} \in \operatorname{span} \lbrace \e_i(t+j),\e_n(t+j)\rbrace_{j=0}^k,
		\end{split}
	\end{align}
	where $\lbrace \e_n(t+j)\rbrace_{j=0}^k$ and $ \lbrace \e_i(t+j),\e_n(t+j)\rbrace_{j=0}^k$ are orthogonal to $\mathcal{H}_{t-}^{\y}$ because $\e$ is the innovation process of $\y$. Projecting \eqref{outeq} onto $\mathcal{H}^{\y_i,\y_j,\y_n}_{t-}$ we can write that
	\begin{equation} \label{Thnii*}
	E_l[\y_j(t+k)|\mathcal{H}^{\y_i,\y_j,\y_n}_{t-}]=C_{j,j}A_{j,j}^{k}\x_j(t)+ C_{j,n}A_{n,n}^{k}\x_n(t) \in \mathcal{H}_{t-}^{\y_j,\y_n},
	\end{equation}
	which leads to (\ref{GC1nii}*). It remained to show that minimality of \eqref{Cor_nn} implies \eqref{Thiv}. By analogy with \eqref{Thnii*} it is easy to see that
	\begin{align} \label{Obseq}
	\begin{split}
	E_l[\y_n(t+k)|\mathcal{H}_{t-}^{\y_n}] &= C_{n,n}A_{n,n}^{k}\x_n(t)\\ E_l[\y_i(t+k)|\mathcal{H}_{t-}^{\y_i,\y_n}] &= C_{i,i}A_{i,i}^{k}\x_i(t)+ C_{i,n}A_{n,n}^{k}\x_n(t),
	\end{split}
	\end{align}
	which gives that $E_l[\mathcal{H}_{t+}^{\y_n}| \mathcal{H}_{t-}^{\y_n}]\subseteq \mathcal{H}^{\x_n}_{t}$ and $E_l[\mathcal{H}_{t+}^{\y_i}|\mathcal{H}_{t-}^{\y_i,\y_n}] \subseteq \mathcal{H}^{\x_i,\x_n}_{t}$. Observe that
	\begin{equation} \label{XVobs}
	E_l[\begin{bmatrix}\y_n(t) \\ \y_n(t+1) \\ \vdots \\ \y_n(t+N-1) \end{bmatrix}|\mathcal{H}_{t-}^{\y_n}] = O_{n} \x_n(t),
	\end{equation}
	where $N$ is the dimension of $A_{n,n}$ and $O_{n}$ is the observability matrix of $(A_{n,n},C_{n,n})$. Since $(A_{n,n},C_{n,n})$ is observable, $O_n$ has left inverse and we can conclude that $E_l[\mathcal{H}_{t+}^{\y_n}| \mathcal{H}_{t-}^{\y_n}]= \mathcal{H}_{t}^{\x_n}$. Using the minimality of the state-space we can partition $\mathcal{H}_{t}^{\x}$ as
	\begin{equation} \label{Xpart}
	\mathcal{H}_{t}^{\x}=(\mathcal{H}_{t}^{\x_1}\ominus \mathcal{H}_{t}^{\x_n}) \oplus \ldots \oplus(\mathcal{H}_{t}^{\x_{n-1}}\ominus \mathcal{H}_{t}^{\x_n}) \oplus \mathcal{H}_{t}^{\x_n},
	\end{equation}
	from which we can conclude that
	\begin{equation} \label{Xdisj}
	(\mathcal{H}_{t}^{\x_i}\ominus \mathcal{H}_{t}^{\x_n})\cap (\mathcal{H}_{t}^{\x_j}\ominus \mathcal{H}_{t}^{\x_n})= \lbrace 0 \rbrace, \hspace{.2cm}  i,j \in \lbrace 1,2, \dots, n-1\rbrace, \hspace{.1cm} i\neq j.
	\end{equation}
	Combining $E_l[\mathcal{H}_{t+}^{\y_n}| \mathcal{H}_{t-}^{\y_n}]=(\mathcal{H}_{t}^{\x_1}\ominus \mathcal{H}_{t}^{\x_n})$ and $E_l[\mathcal{H}_{t+}^{\y_i}|\mathcal{H}_{t-}^{\y_i,\y_n}] \subseteq (\mathcal{H}_{t}^{\x_1}\ominus \mathcal{H}_{t}^{\x_i,\x_n})$ with \eqref{Xdisj} leads to \eqref{Thiv}.\\ \vspace{.25cm}
	
	\textbf{If} In view of Lemma \ref{Lem} from (\ref{GC1ni}*) the existence of a state-space representation of the form \eqref{Cor_nn} with the properties (\ref{Thi})-(\ref{Thiii}) follows. Besides, (\ref{GC1nii}*) ensures the process $\e$ to be the innovation process of $\y$. Therefore it remains to show that \eqref{Thiv} implies the minimality of such a representation. From \cite[Theorem 6.5.4]{LinPic15} we know that the linear stochastic system $(A,K,C,I)$ is minimal if and only if $(A,C)$ is observable, $(A,K)$ is reachable and the state-space is constructible. In Lemma \ref{Lem} observability and constructibility have already been proved thus we only need to show that reachability holds as well.
	
	According to \cite[Proposition 6.1.1]{LinPic15} $(A,K)$ is reachable if and only if $\x(t)$ is a base of $\mathcal{H}^{\x_n}_t$. Notice that since the representations (\ref{reprin}) are minimal we already know that $\mathcal{H}^{\x_i}_t\cap \mathcal{H}^{\x_n}_t=\lbrace0 \rbrace$. For reachability we need to see that $\dim(\mathcal{H}^{\x}_t)=\sum_{k=1}^n \dim(\mathcal{H}^{\x_i}_t)$.\footnote{Dimension means that the number of one-dimensional processes in a base.} To this end we will see that \eqref{Thiv} implies the relation
	\begin{equation} \label{AssX}
	(\mathcal{H}_{t}^{\x_i}\ominus \mathcal{H}_{t}^{\x_n})\cap (\mathcal{H}_{t}^{\x_j}\ominus \mathcal{H}_{t}^{\x_n})= \lbrace 0 \rbrace, \hspace{.2cm}  i,j \in \lbrace 1,2, \dots, n-1\rbrace, \hspace{.1cm} i\neq j,
	\end{equation}
	which enables $\mathcal{H}^{\x}_t$ to be written as in \eqref{Xpart}. Similarly as in the sufficient part of the proof it can be seen that the representation obtained from (\ref{GC1ni}*), by applying Lemma \ref{Lem}, leads to
	$$ E_l[\mathcal{H}^{\y_n}_{t+}|\mathcal{H}^{\y_n}_{t-}]=\mathcal{H}^{\x_n}_t, \hspace{.3cm} _lE[\mathcal{H}^{\y_i}_{t+}|\mathcal{H}_{t-}^{\y_i,\y_n}]\subseteq \mathcal{H}^{\x_i,\x_n}_t.$$
	As $O_i^{+}$ is of full row rank, we obtain from \eqref{XVobs} that $E_l[\mathcal{H}^{\y_i}_{t+}|\mathcal{H}_{t-}^{\y_i,\y_n}]+\mathcal{H}^{\x_n}_t \supseteq \mathcal{H}^{\x_i}_t$. Then, it follows that
	$$ E_l[\mathcal{H}^{\y_n}_{t+}|\mathcal{H}^{\y_n}_{t-}]=\mathcal{H}^{\x_n}_t, \hspace{.3cm}E_l[\mathcal{H}^{\y_i}_{t+}|\mathcal{H}_{t-}^{\y_i,\y_n}]+E_l[\mathcal{H}^{\y_n}_{t+}|\mathcal{H}^{\y_n}_{t-}] = \mathcal{H}^{\x_i,\x_n}_t.$$
	These relations together with \eqref{Thiv} lead to \eqref{AssX}. From $\mathcal{H}^{\x_i}_t\cap \mathcal{H}^{\x_n}_t=\lbrace0 \rbrace$, it follows that the dimension of $\mathcal{H}^{\x}_t$ is equal to $\sum_{i=1}^n \dim(\mathcal{H}^{\x_i}_t)$. This means that the representation is reachable and, conclusively, it is minimal.
\end{proof}

\begin{proof}[Proof of Lemma \ref{Algo34}]
	In the proof of Lemma \ref{Lem} we showed that if $\y_i$ does not Granger cause $\y_n$ for $i \in \lbrace 1, \ldots,n-1 \rbrace$ then the representations given by Theorem \ref{Kalman22} for $[\y_i^T,\y_n^T]^T$ can be combined into a coordinated form. Therefore, according to Lemma \ref{Algo12} step 1 in Algorithm \ref{algoCoordABCD}, \ref{algoCoordY} and step 2-3 in Algorithm \ref{algoCoordABCD} produces system \eqref{repr1n_Lem} in Lemma \ref{Lem}. In the proof of Theorem \ref{GCnn} we showed that this system exists and is in coordinated form with the properties \eqref{Thi}-\eqref{Thii}-\eqref{Thiii} in Theorem \ref{GCnn} if and only if the Granger non-causality conditions 1. and 2. hold. Since observability and constructibility is a consequence of the construction then the system is minimal if controllability holds. From the proof of Theorem \ref{GCnn} it turns out that minimality holds if and only if \eqref{Thiv} is fulfilled.
\end{proof}

\section*{Conclusion}
Granger non-causality between two processes ($\y_1,\y_2$) is usually inspected by looking at the coefficient matrices of an MA or AR representation of their joint process ($\y=[\y_1^T,\y_2^T]^T$). In this paper we showed another way for investigating Granger causality, namely by looking at whether a Kalman representation of $\y$, chosen in a certain way, is in block triangular form. As a result, in the presence of Granger non-causality this method provides a state-space model in block triangular form which can be calculated algorithmically. In fact, this approach turned out to be useful for constructing well structured state-space model characterizing a leader-follower interconnection structure specified by Granger non-causalities in $\y$. A class of sate-space models in a specific form, called coordinated form, was introduced analogue to the deterministic terminology. In our main result a state-space representation for $\y$ in coordinated form was proved to characterize Granger non-causality conditions of a coordinated (leader-follower) interconnection structure. This result is built on the results for the state-space characterization of a simple Granger non-causality, therefore it can be constructed algorithmically.

The state-space model in coordinated form characterizing Granger non-causality conditions can be calculated from data. However, its statistical behaviour need to be further studied in future research. Also, the generalization of this result either for non-linear models or for more complex interconnection structure remains for future work.

\bibliographystyle{plain}
\bibliography{jozsa}

\begin{thebibliography}{10}

\bibitem{LinPic15}
G.~Picci A.~Lindquist.
\newblock {\em Linear Stochastic Systems}, volume~1.
\newblock Springer-Verlag Berlin Heidelberg, 2015.

\bibitem{Caines76}
P.E. Caines.
\newblock Weak and strong feedback free processes.
\newblock {\em Automatic Control, IEEE Transactions on}, 21(5):737--739, Oct
  1976.

\bibitem{CaiCha75}
P.E. Caines and C.~Chan.
\newblock Feedback between stationary stochastic processes.
\newblock {\em Automatic Control, IEEE Transactions on}, 20(4):498--508, Aug
  1975.

\bibitem{Caines09}
P.E. Caines, R.~Deardon, and H.P. Wynn.
\newblock Bayes nets of time series: Stochastic realizations and projections.
\newblock In Luc Pronzato and Anatoly Zhigljavsky, editors, {\em Optimal Design
  and Related Areas in Optimization and Statistics}, volume~28 of {\em Springer
  Optimization and Its Applications}, pages 155--166. Springer New York, 2009.

\bibitem{CaDeWy03}
Peter~E. Caines, R.~Deardon, and H.P. Wynn.
\newblock Conditional orthogonality and conditional stochastic realization.
\newblock In Anders Rantzer and Christopher~I. Byrnes, editors, {\em Directions
  in Mathematical Systems Theory and Optimization}, volume 286, pages 71--84.
  Springer Berlin Heidelberg, 2003.

\bibitem{CaiWyn07}
Peter~E. Caines and Henry~P. Wynn.
\newblock An algebraic framework for bayes nets of time series.
\newblock In Alessandro Chiuso, Stefano Pinzoni, and Augusto Ferrante, editors,
  {\em Modeling, Estimation and Control}, volume 364 of {\em Lecture Notes in
  Control and Information Sciences}, pages 45--57. Springer Berlin Heidelberg,
  2007.

\bibitem{Eichler12}
M.~Eichler.
\newblock Graphical modelling of multivariate time series.
\newblock {\em Probability Theory and Related Fields}, 153(1):233--268, 2012.

\bibitem{GevAnd82}
M.R. Gevers and B.~Anderson.
\newblock On jointly stationary feedback-free stochastic processes.
\newblock {\em Automatic Control, IEEE Transactions on}, 27(2):431--436, Apr
  1982.

\bibitem{Granger63}
C.W.J. Granger.
\newblock Economic processes involving feedback.
\newblock {\em Information and Control}, 6(1):28--48, 1963.

\bibitem{Hsiao82}
C.~Hsiao.
\newblock Autoregressive modeling and causal ordering of econometric variables.
\newblock {\em Journal of Economic Dynamics and Control}, 4:243--259, 1982.

\bibitem{DufRen98}
Eric~Renault Jean-Marie~Dufour.
\newblock Short run and long run causality in time series: Theory.
\newblock {\em Econometrica}, 66(5):1099--1125, 1998.

\bibitem{julius-cp}
A.~Agung Julius, Michael Zavlanos, Stephen Boyd, and George~J. Pappas.
\newblock Genetic {N}etwork {I}dentification {U}sing {C}onvex {P}rogramming.
\newblock {\em Systems {B}iology, {IET}}, 3:155--166, 2009.

\bibitem{sontag_network3}
T.~Kang, R.~Moore, Y.~Li, E.D. Sontag, and L.~Bleris.
\newblock Discriminating direct and indirect connectivities in biological
  networks.
\newblock {\em Proc Natl Acad Sci USA}, 112:12893--12898, 2015.

\bibitem{Kempker12}
P.~L. Kempker.
\newblock {\em Coordination Control of Linear Systems}.
\newblock PhD thesis, Amsterdam: Vrije Universiteit, 2012.

\bibitem{KeRaSch14}
Pia~L. Kempker, Andr\'{e}~C.M. Ran, and Jan~H. van Schuppen.
\newblock Construction and minimality of coordinated linear systems.
\newblock {\em Linear Algebra and its Applications}, 452:202--236, 2014.

\bibitem{BarSet15}
Anil K.~Seth Lionel~Barnett.
\newblock Granger causality for state space models.
\newblock {\em Physical Review E}, 91(4):737--739, Apr 2015.

\bibitem{nordling}
Torbj\"{o}rn~E.M. Nordling and Elling~W. Jacobsen.
\newblock On {S}parsity as a {C}riterion in {R}econstructing {B}iochemical
  {N}etworks.
\newblock In {\em 18th IFAC World Congress}, 2011.

\bibitem{Pambakian11}
Nicola Pambakian.
\newblock Lqg coordination control.
\newblock Master's thesis, Delft University of Technology, 2011.

\bibitem{JHvanSchuppen2014}
Andr\'e C.~M. Ran and Jan~H. Van~Schuppen.
\newblock Coordinated linear systems.
\newblock In {\em Coordination Control of Distributed Systems}, volume 456 of
  {\em Lecture Notes in Control and Information Sciences}, pages 113--121,
  2014.

\bibitem{roebroeck2011}
Alard. Roebroeck, Anil~K. Seth, and Pedro~A. Valdes-Sosa.
\newblock Causal time series analysis of functional magnetic resonance imaging
  data.
\newblock {\em Journal of Machine Learning Research, Proceedings Track},
  12:65--94, 2011.

\bibitem{sosa2011}
Pedro~A. Valdes-Sosa, Alard Roebroeck, Jean Daunizeau, and Karl~J. Friston.
\newblock Effective connectivity: Influence, causality and biophysical
  modeling.
\newblock {\em Neuroimage}, 58(2):339--361, 2011.

\bibitem{GongalvesRecent2015}
Y.~Yuan, K~Glover, and J~Gonçalves.
\newblock On minimal realisations of dynamical structure functions.
\newblock {\em Automatica}, 55:159--164, 2015.

\bibitem{Yuan20111230}
Ye~Yuan, Guy-Bart Stan, Sean Warnick, and Jorge Goncalves.
\newblock Robust dynamical network structure reconstruction.
\newblock {\em Automatica}, 47(6):1230 -- 1235, 2011.
\newblock Special Issue on Systems Biology.

\end{thebibliography}
\end{document}